\documentclass[11pt,a4paper]{article}
\usepackage{amsmath,amssymb,amsthm,amsfonts,latexsym,graphicx,subfigure}
\usepackage[affil-it]{authblk}
\usepackage[numbers,sort&compress]{natbib}
\usepackage[a4paper,margin=3cm]{geometry}
\usepackage{xcolor}

\usepackage{hyperref}

\newcommand{\gm}{\mathop{\rm gm}\nolimits}
\newcommand{\Hsm}{\mathop{\rm Hsm}\nolimits}
\DeclareMathOperator{\id}{Id}
\DeclareMathOperator{\spec}{Spec}
\DeclareMathOperator{\hspec}{Hspec}
\DeclareRobustCommand{\stirling}{\genfrac\{\}{0pt}{}}

\newcommand{\N}{\mathbb{N}}

\newcommand{\R}{\mathbb{R}}
\newcommand{\C}{\mathbb{C}}
\newcommand{\LL}{\mathcal{L}}

\theoremstyle{plain}
\newtheorem{theorem}{Theorem}[section]
\newtheorem{lemma}[theorem]{Lemma}
\newtheorem{proposition}[theorem]{Proposition}
\newtheorem{corollary}[theorem]{Corollary}

\theoremstyle{definition}
\newtheorem{definition}[theorem]{Definition}
\newtheorem{example}[theorem]{Example}
\theoremstyle{remark}
\newtheorem{remark}[theorem]{Remark}

\usepackage[affil-it]{authblk}

\begin{document}
\bibliographystyle{plain} 
\title{Spectral theory of weighted hypergraphs via tensors}
\author[1]{Francesco Galuppi\footnote{Email address: fgaluppi@impan.pl}}
\author[2,3]{Raffaella Mulas\footnote{Email address: raffaella.mulas@mis.mpg.de}}
\author[4]{Lorenzo Venturello \footnote{Email address: lven@kth.se}}
\affil[1]{Institute of Mathematics of the Polish Academy of Sciences, Warsaw, Poland}
\affil[2]{The Alan Turing Institute, London, UK}
\affil[3]{University of Southampton, Southampton, UK}
\affil[4]{KTH Royal Institute of Technology, Stockholm, Sweden}
\date{}
\maketitle

\allowdisplaybreaks[4]

	\begin{abstract}
One way to study a hypergraph is to attach to it a tensor. Tensors are a generalization of matrices, and they are an efficient way to encode information in a compact form. %Just as matrices, tensors have eigenvalues and eigenvectors.
In this paper we study how properties of weighted hypergraphs are reflected on eigenvalues and eigenvectors of their associated tensors. We also show how to efficiently compute eingenvalues with some techniques from numerical algebraic geometry.\newline

\noindent {\bf Keywords:} Spectral hypergraph theory, Tensors, Weighted hypergraphs, Eigenvalues
	\end{abstract}

\section{Introduction}\label{sec: intro}
\subsection*{Spectral hypergraph theory}
Spectral graph theory is the study of the qualitative properties of a graph $G=(V,E)$ that can be inferred from the \emph{spectrum}, i.e.\ the multiset of the eigenvalues, of a square matrix associated to $G$. If the vertices of $G$ are labelled as $v_1,\ldots,v_N$, the considered square matrices are usually the $N\times N$ \emph{adjacency matrix} $A$, whose entries are
\begin{equation*}
   A_{ij}:=\begin{cases}1 & \text{if }(v_i,v_j)\in E\\
   0 & \text{otherwise,}
    \end{cases}
\end{equation*}the \emph{Kirchhoff Laplacian} 
\begin{equation*}
    K:=D-A,
\end{equation*}
where $D$ is the diagonal matrix of the degrees, and the \emph{normalized Laplacian}
\begin{equation*}
    L:=\id-D^{-1/2}AD^{-1/2},
\end{equation*}
 where $\id$ is the $N\times N$ identity matrix. There is a $1:1$ correspondence between graphs and each of these operators. While there exist non-isomorphic graphs with the same spectra, nevertheless such spectra are known to detect many important geometric properties of the graph. Thus, if two graphs are isospectral with respect to a given operator, they have similar structures. Spectral graph theory has a long history, it has many connections with the spectral theory in Riemannian geometry, and it is widely used in applications. In fact, because empirical networks can often be modeled as graphs, the computation of the spectra is a  valuable  tool for studying and analyzing real data. Moreover, this theory also finds applications in the study of dynamical systems on graphs. The normalized Laplacian is isospectral to the \emph{random walk Laplacian}
\begin{equation*}
    \mathcal{L}:=\id-D^{-1}A,
\end{equation*}whose off-diagonal entry
\begin{equation*}
    \mathcal{L}_{ij}=-\frac{A_{ij}}{\deg v_i}%=-\mathbb{P}(\text{random walker goes from }v_i \text{ to } v_j),
\end{equation*}
is minus  the probability that a random walker goes from $v_i$ to $v_j$. Hence, the spectral theory of $L$ and $\mathcal{L}$ has many connections with random walks on graphs. We refer the reader to \cite{chung,brouwer2011spectra} for classical monographs on the spectral theory of graphs.\newline

Hypergraphs are defined as a generalization of graphs in which edges can link more than two vertices, that is, edges are \emph{sets} of vertices of any cardinality. They allow the modeling of many more real networks than graphs, therefore they are often used in applications as well, see for instance \cite{entropy,BodoKatonaSimon,klamt2009hypergraphs,bitcoin,ZhangLiu,lanchier2013stochastic}. It is natural to ask what is the best way to generalize spectral graph theory to the case of hypergraphs, but the answer is not univocal. In fact, the spectral theory of hypergraphs can be studied either via matrices or via tensors, and the two approaches have different advantages. While there is a $1:1$ correspondence between hypergraphs and their associated tensors, the same does not hold for the hypergraph adjacency and Laplacian matrices. Hence, the spectrum of a hypergraph tensor is expected to detect more precise structural properties of the hypergraph than the spectrum of a hypergraph matrix. However, the tensor eigenvalue problem is NP-hard \cite{NPhard}, therefore it is more convenient to consider matrices for the analysis of big data that are modeled with hypergraphs.\newline

In this work, we bring forward the spectral theory of hypergraphs via tensors and we generalize the operators in \cite{Banerjee,Parui,Ouvrard,Ouvrard2,BanerjeeChar,signless} to the case of weighted hypergraphs, for which each edge has a positive weight. We show that most of the spectral properties that are known for unweighted hypergraphs can be generalized to the weighted case, and we prove several results which are also new for the unweighted case. We refer to \cite{QL} for a monograph on the spectral theory via tensors of uniform, unweighted hypergraphs. We refer to \cite{Multilayer} for the related spectral theory of multilayer networks. For completeness, we also refer the reader to \cite{MKJ,Reff2014,ReffRusnak,MulasZhang,JM2019,JM2020,Lucas2013} for a vast\,---\,but by no means complete\,---\,literature on the spectral theory of hypergraphs via matrices.

\subsection*{Tensors}
We start by recalling several definitions and properties of tensors that will be needed throughout the paper. We refer to \cite{QL%,landsberg
} for a monograph on this topic. Given $n\in\N$,
we denote by $[n]$ the set $\{1,\ldots,n\}$. We indicate a vector in $\C^n$ by ${\bf x}=(x_1,\dots,x_n)$. We write ${\bf x}\ge 0$ if $x_i\ge 0$ for every $i\in[n]$. \newline
Let $k,n\geq 2$. A $k$-th order $n$-dimensional \emph{tensor} $T$ consists of $n^k$ complex entries
\begin{equation*}
    T_{i_1,\ldots,i_k}\in \mathbb{C},
\end{equation*}
where $i_1,\dots,i_k\in[n]$. The tensor $T$ is \emph{symmetric} if its entries are invariant under any permutation of their indices.\newline

One of the many differences between the spectral theory of matrices and of tensors is that a the eigenvalues of a real symmetric tensor do not need to be real. %For this reason we will often work on $\C$.

Given a vector $\mathbf{x}\in\mathbb{C}^n$, define $T\mathbf{x}^{k-1}\in\mathbb{C}^n$ by
\begin{equation*}
    (T\mathbf{x}^{k-1})_i=\sum_{i_2,\ldots,i_k\in [n]} T_{i,i_2,\ldots,i_k}x_{i_2}\cdots x_{i_k}.
\end{equation*}Let $\mathbf{x}^{[k-1]}\in \mathbb{C}^n$ be the vector with entries $x^{[k-1]}_i:=x^{k-1}_i$. If $T\mathbf{x}^{k-1}=\lambda \mathbf{x}^{[k-1]}$ for some $\lambda\in\C$ and some non-zero vector $\mathbf{x}\in\C^n$, then we say that $(\lambda,\mathbf{x})$ is an \emph{eigenpair} for $T$. The number $\lambda$ is an \emph{eigenvalue} of $T$ and $\mathbf{x}$ is an \emph{eigenvector}. %In this case, 
The \emph{spectral radius} of $T$, denoted $\rho(T)$, is the largest modulus of the eigenvalues of $T$. \newline
Let $\id$ denote the $k$-th order $n$-th dimensional \emph{unit tensor}, with entries
\begin{equation*}
    \id_{i_1,\ldots,i_k}:=\begin{cases}1 & \text{ if }i_1=\ldots=i_k\\
    0 & \text{ otherwise.}\end{cases}
    \end{equation*}
 Just as matrices, tensors have a determinant. As illustrated in \cite[Section 2.1.3]{QL}, the determinant of a tensor $T$ is the resultant of the system of equations $T\mathbf{x}^{k-1}=0$. It is a polynomial in the entries of $T$ that vanishes if and only if the system has a non-zero solution.
 As shown in \cite[Theorem 2.12]{QL},  the eigenvalues of $T$ are the roots of the \emph{characteristic polynomial}
\begin{equation*}
    \varphi_T(\lambda):=\det(T-\lambda\id)
\end{equation*}
 of $T$. The \emph{spectrum} of $T$, denoted $\spec(T)$, is the multiset of its eigenvalues, counted with multiplicity as roots of the characteristic polynomial. An eigenpair $(\lambda,\mathbf{x})$ is given by an \emph{$H$-eigenvalue} and an \emph{$H$-eigenvector}, respectively, if $\lambda\in \mathbb{R}$ and $\mathbf{x}\in\mathbb{R}^n$. We let $\hspec(T)$ denote the set of distinct $H$-eigenvalues of $T$%, that does not account for multiplicities
.\newline
The tensor $T$ is \emph{non-negative} if all its entries are non-negative. The tensor $T$ is \emph{weakly irreducible} %\cite{QL}
if, for any non-empty proper index subset $J$ of $[n]$, there is at least one entry
\begin{equation*}
    T_{i_1,\ldots,i_k}\neq 0,
\end{equation*}where $i_1\in J$ and at least one index $i_j\in [n]\setminus J$, for $j=2,\ldots,k$.\newline 
Similarly, $T$ is \emph{reducible} if there exists a non-empty proper set $J$ of $[n]$ such that  
\begin{equation*}
    T_{i_1,\ldots,i_k}= 0,
\end{equation*}for each $i_1\in J$, $i_2,\ldots,i_k\in [n]\setminus J$. It is \emph{irreducible} if it is not reducible.\newline
In \cite[Theorems 3.25 and 3.26]{QL},  the Perron-Frobenius theorem for non-negative matrices has been generalized to weakly irreducible non-negative tensors.
\begin{theorem}\label{minmax:waklyirr}
If $T$ is non-negative and weakly irreducible, then 
\begin{equation*}
    \rho(T)=\max_{\substack{\mathbf{x}\geq 0\\ \mathbf{x}\neq 0} }\min_{x_i>0}\frac{(T\mathbf{x}^{k-1})_i}{x_i^{k-1}}
\end{equation*}is a positive $H$-eigenvalue, with a positive $H$-eigenvector $\bar{\mathbf{x}}$. Furthermore, $\rho(T)$ is the unique $H$-eigenvalue of $T$ with a positive $H$-eigenvector, and $\bar{\mathbf{x}}$ is the unique positive $H$-eigenvector associated to $\rho(T)$, up to a multiplicative constant.
\end{theorem}
Now, $T$ is \emph{diagonally dominated} if, for each $i\in [n]$,
\begin{equation}\label{eq:def-dd}
    T_{i,\ldots,i}\geq \sum_{\substack{i_2,\ldots,i_k\in [n]\\ \text{not all equal to }i}}\left|T_{i,i_2,\ldots,i_k}\right|.
\end{equation}By \cite[Theorem 2]{Song}, if $T$ is  any tensor  and $\lambda$ is an eigenvalue for $T$, then there exists $i\in [n]$ such that 
\begin{equation}\label{eq:dd}
    |\lambda-T_{i,\ldots,i}|\leq \sum_{\substack{i_2,\ldots,i_k\in [n]\\ \text{not all equal to }i}}\left|T_{i,i_2,\ldots,i_k}\right|.
\end{equation}An immediate consequence is the following
\begin{corollary}\label{cor:dd}
If $T$ is a diagonally dominated tensor, then all its $H$-eigenvalues are non-negative.
\end{corollary}

\subsection*{Eigenvalue multiplicities}
For a given tensor, the \emph{algebraic multiplicity} of an eigenvalue $\lambda$, denoted $\textrm{am}(\lambda)$, is its multiplicity as a root of the characteristic polynomial. The \emph{geometric multiplicity} of $\lambda$, denoted $\textrm{gm}(\lambda)$, is the dimension of its \emph{eigenvariety}, i.e.\ the variety of its eigenvectors. While algebraic and geometric multiplicity coincide for symmetric matrices, this is not always the case for symmetric tensors.  In \cite[Conjecture 1.1]{multiplicities}, it is conjectured that, for any $\lambda$, 
\begin{equation*}\label{conj1.1}
 \textrm{am}(\lambda)\geq \textrm{gm}(\lambda)(k-1)^{\textrm{gm}(\lambda)-1}.
\end{equation*}Given an eigenvalue $\lambda$, we also introduce its \emph{span multiplicity}, denoted $\textrm{sm}(\lambda)$, as the dimension of the vector space
\begin{equation*}
    \textrm{span}(\{\mathbf{x}\in \mathbb{C}^n\,:\, \mathbf{x} \text{ eigenvector for }\lambda\}).
\end{equation*}Clearly, $\textrm{sm}(\lambda)\geq \textrm{gm}(\lambda)$.  In \cite[Section 3]{Qi} it is conjectured that, for any $\lambda$,
\begin{equation*}
     \textrm{am}(\lambda)\geq  \textrm{sm}(\lambda).
\end{equation*}
If $\lambda$ is a real eigenvalue, we introduce its \emph{Hspan multiplicity}, denoted $\textrm{Hsm}(\lambda)$, as the dimension of the vector space
\begin{equation*}
    \textrm{span}(\{\mathbf{x}\in \mathbb{R}^n\,:\, \mathbf{x} \text{ eigenvector for }\lambda\}).
\end{equation*}In the case of symmetric matrices, every eigenvalue $\lambda$ is an $H$-eigenvalue and
\begin{equation*}
\textrm{am}(\lambda) = \textrm{sm}(\lambda) = \textrm{Hsm}(\lambda) = \textrm{gm}(\lambda).
\end{equation*}
In spectral graph theory and spectral hypergraph theory via matrices, the multiplicity of an eigenvalue is often studied by counting the maximum number of its linearly independent eigenvectors. This is what motivates us to study the span and the Hspan multiplicities.

\subsection*{Stirling numbers}
Given a multiset $i_1,\ldots,i_k$, let $\{i_1,\ldots,i_k\}$ be the set obtained from $i_1,\ldots,i_k$ by not accounting for multiplicity. Given a $k$-th order $n$-dimensional tensor $T$ and an index set $J\subseteq [n]$ of cardinality $r$, we say that an entry $T_{i_1,\ldots,i_k}$ of $T$ \emph{corresponds to} $J$ if $\{i_1,\ldots,i_k\}=J$. We let
\begin{equation*}
    \mathcal{N}(r,k):=|\text{entries of $T$ corresponding to }J|.
\end{equation*}
Given $j\in J$, we define its \emph{$j$-th row} as the $(k-1)$-order $n$-dimensional tensor $T_j$ obtained by setting the first index of $T$ equal to $j$. We also set
\begin{equation*}
 N(r,k):=\frac{\mathcal{N}(r,k)}{r}=|\text{entries of the row $T_j$ corresponding to }J|.
\end{equation*}
Observe that it does not depend on $j$. As shown in \cite{Ouvrard},
\begin{equation*}
    \mathcal{N}(r,k)= \sum_{\substack{k_1,\ldots,k_r\geq 1,\\ \sum_{j}k_j=k}}\frac{k !}{k_1!\cdots k_r!},
\end{equation*}from which it follows that
\begin{equation*}
    N(r,k)=\frac{1}{r}\cdot\left( \sum_{\substack{k_1,\ldots,k_r\geq 1,\\ \sum_{j}k_j=k}}\frac{k !}{k_1!\cdots k_r!}\right).
\end{equation*} By Proposition 5.5 in \cite{mariconda},  we can write $\mathcal{N}(r,k)=\stirling{k}{r}r!$, where 
\begin{equation*}
    \stirling{k}{r}:=\frac{1}{r!}\sum_{j=0}^r(-1)^j\binom{r}{j}(r-j)^k
\end{equation*}
 is the \emph{Stirling number of the second kind}.  This can be seen from  the fact that $ \stirling{k}{r}$ counts the number of $r$-partitions of a set of cardinality $k$, while $r!$ counts the number of permutations of $r$ objects. Hence,
 \begin{equation*}
      N(r,k)=\stirling{k}{r}(r-1)!.
\end{equation*}In particular, $N(k,k)=(k-1)!$. The above characterization of $N(r,k)$ allows us to give a description of the hypergraph tensors which is simpler than the one in \cite{Banerjee,Parui,Ouvrard,Ouvrard2,BanerjeeChar}.
\subsection*{Weighted hypergraphs}
Throughout the paper, we fix a weighted hypergraph $G=(V,E,w)$ with vertex set $V=\{v_1,\ldots,v_N\}$, edge set $E=\{e_1,\ldots,e_M\}$ and weight function $w:E\rightarrow \mathbb{R}_{>0}$. We assume that each edge contains at least two vertices.\newline % and that $e_i\neq e_j$ if $i\neq j$.
Given $v\in V$, its \emph{degree} is
\begin{equation*}
\deg v:=\sum_{e\in E:\, v\in e}w(e).
\end{equation*}Given $e\in E$, we denote its cardinality by $|e|$. We let
\begin{equation*}
 \delta:=\min_{v\in V}\deg v,
 \quad \Delta:=\max_{v\in V}\deg v  \quad\text{and}\quad \nabla:=\max_{e\in E}|e|.
\end{equation*} $G$ is $\Delta$-\emph{regular} if $\deg v=\Delta$ for each $v\in V$; it is $\nabla$-\emph{uniform} if $|e|=\nabla$ for each $e\in E$.\newline
Given $r\leq \nabla$, we let 
\begin{equation*}
    E_r:=\{e\in E:\, |e|=r\}.
\end{equation*}
$G$ is \emph{connected} if, for every pair of vertices $v,w\in V$, there exists a \emph{path} that connects $v$ and $w$, i.e.\ there exist $\hat{v}_1,\ldots,\hat{v}_k\in V$ and $\hat{e}_1,\ldots,\hat{e}_{k-1}\in E$ such that $\hat{v}_1=v$, $\hat{v}_k=w$, and $\{\hat{v}_i,\hat{v}_{i+1}\}\subseteq \hat{e}_i$ for each $i=1,\ldots,k-1$.\newline
The hypergraph $G=(V,E,w)$ is \emph{unweighted} if $w(e)=1$ for each $e\in E$. In this case, we use the notation $G=(V,E)$.\newline
A \emph{simple graph} is an unweighted $2$-uniform hypergraph.

\subsection*{Structure of the paper}
In Section \ref{section:hyp} we introduce the tensors we are interested in. Our definitions generalize the known hypergraph tensors to the case of weighted hypergraphs. In Section \ref{sec: first properties} we prove some basic properties of the eigenvalues of such tensors. Section \ref{section: duplicate} is devoted to duplicate vertices, a combinatorial feature of a hypergraph that has a clear impact on the spectra. This is particularly evident for \emph{hyperflowers}, a remarkable family of hypergraphs that generalize star graphs and that we analyse in Section \ref{section: iperfiore}. Section \ref{section: symmetries} adds more words to our dictionary. We are able to translate properties of the hypergraphs, such as being bipartite or colorable, to spectral symmetries. Finally, in Section \ref{sec: computations} we show how to effectively compute eigenvalues and their geometric multiplicities. We apply techniques from numerical algebraic geometry.

\section{Hypergraph tensors}\label{section:hyp}
As we recalled in the introduction, there are several tensors attached to a hypergraph. In this section we generalize these tensors for weighted hypergraphs.\newline
Let $G=(V,E,w)$ be a weighted hypergraph on $N$ nodes, with largest edge cardinality $\nabla$. The \emph{adjacency tensor} of $G$ is the $\nabla$-th order $N$-dimensional tensor $A=A(G)$ with entries
\begin{equation*}
A_{i_1,\ldots,i_\nabla}:=\begin{cases}
0 & \text{ if }\{v_{i_1},\ldots,v_{i_\nabla}\}\notin E\\
\frac{w(e)}{N(r,\nabla)} &  \text{ if }\{v_{i_1},\ldots,v_{i_\nabla}\}=e\in E_r.\\
\end{cases}
\end{equation*}
The \emph{Kirchhoff Laplacian tensor} of $G$ is the $\nabla$-th order $N$-dimensional tensor $K=K(G)$ with entries
\begin{equation*}
K_{i_1,\ldots,i_\nabla}:=\begin{cases}
\deg v_{i_1} & \text{ if }i_1=\ldots=i_\nabla\\
-A_{i_1,\ldots,i_\nabla} & \text{ otherwise. }
\end{cases}
\end{equation*}
The \emph{normalized Laplacian tensor} of $G$ is the $\nabla$-th order $N$-dimensional tensor $L=L(G)$ with entries
\begin{equation*}
L_{i_1,\ldots,i_\nabla}:=\begin{cases}
1 & \text{ if }i_1=\ldots=i_\nabla\\
-A_{i_1,\ldots,i_\nabla}\cdot\prod_{j\in\{i_1,\ldots,i_\nabla\}} \frac{1}{ \sqrt[\nabla]{\deg v_{j}}} & \text{ otherwise. }
\end{cases}
\end{equation*}
The \emph{random walk Laplacian tensor} of $G$ is the $\nabla$-th order $N$-dimensional tensor $\mathcal{L}=\mathcal{L}(G)$ with entries
\begin{equation*}
\mathcal{L}_{i_1,\ldots,i_\nabla}:=\begin{cases}
1 & \text{ if }i_1=\ldots=i_\nabla\\
-\frac{A_{i_1,\ldots,i_\nabla}}{\deg v_{i_1}} & \text{ otherwise. }
\end{cases}
\end{equation*}

For unweighted hypergraphs, $A$, $K$, $L$ and $\mathcal{L}$ coincide with the tensors in \cite[Section 3]{Banerjee}. In the case of simple graphs, the adjacency, Kirchhoff Laplacian, normalized Laplacian and random walk Laplacian tensors coincide with the adjacency, Kirchhoff Laplacian, normalized Laplacian and random walk Laplacian matrices, respectively.\newline
Some of the tensors defined above have a signless version. The \emph{signless Kirchhoff Laplacian tensor} of $G$ is the $\nabla$-th order $N$-dimensional tensor $K^+=K^+(G)$ with entries
\begin{equation*}
K^+_{i_1,\ldots,i_\nabla}:=\begin{cases}
\deg v_{i_1} & \text{ if }i_1=\ldots=i_\nabla\\
A_{i_1,\ldots,i_\nabla} & \text{ otherwise. }
\end{cases}
\end{equation*}The \emph{signless normalized Laplacian tensor} of $G$ is the $\nabla$-th order $N$-dimensional tensor $L^+=L^+(G)$ with entries
\begin{equation*}
L^+_{i_1,\ldots,i_\nabla}:=\begin{cases}
1 & \text{ if }i_1=\ldots=i_\nabla\\
A_{i_1,\ldots,i_\nabla}\cdot\prod_{j\in\{i_1,\ldots,i_\nabla\}} \frac{1}{ \sqrt[\nabla]{\deg v_{j}}} & \text{ otherwise. }
\end{cases}
\end{equation*}The \emph{signless random walk Laplacian tensor} of $G$ is the $\nabla$-th order $N$-dimensional tensor $\mathcal{L}^+=\mathcal{L}^+(G)$ with entries
\begin{equation*}
\mathcal{L}^+_{i_1,\ldots,i_\nabla}:=\begin{cases}
1 & \text{ if }i_1=\ldots=i_\nabla\\
\frac{A_{i_1,\ldots,i_\nabla}}{\deg v_{i_1}} & \text{ otherwise. }
\end{cases}
\end{equation*}

 To the best of our knowledge, the definitions of $L^+$ and $\mathcal{L}^+$ are new also for the case of uniform unweighted hypergraphs.  For unweighted hypergraphs,  $K^+$ coincides with the one in \cite{signless}.  In the case of simple graphs, the signless Kirchhoff Laplacian, signless normalized Laplacian and signless random walk Laplacian tensors coincide with the signless Kirchhoff Laplacian, signless normalized Laplacian and signless random walk Laplacian matrices, respectively.

\begin{remark}
From the definition, it is apparent that the tensors $A$, $K$, $K^+$, $L$ and $L^+$ are symmetric and that the tensors $A$, $K^+$, $L^+$ and $\mathcal{L}^+$ are non-negative.
\end{remark}

\begin{example}\label{ex:123}
Let $G=(V,E,w)$ be the weighted hypergraph with vertex set $V=\{v_1,v_2,v_3\}$, edge set $E = \{\{v_1,v_2\},\{v_1,v_2,v_3\}\}$ and weights $w(\{v_1,v_2\})=1$ and $w(\{v_1,v_2,v_3\})=2$.%such that %(Figure \ref{fig:ex123}):
%\begin{itemize}   \item     \item ,    \item \end{itemize}
\begin{figure}[ht]
    \centering
    \includegraphics[width=7cm]{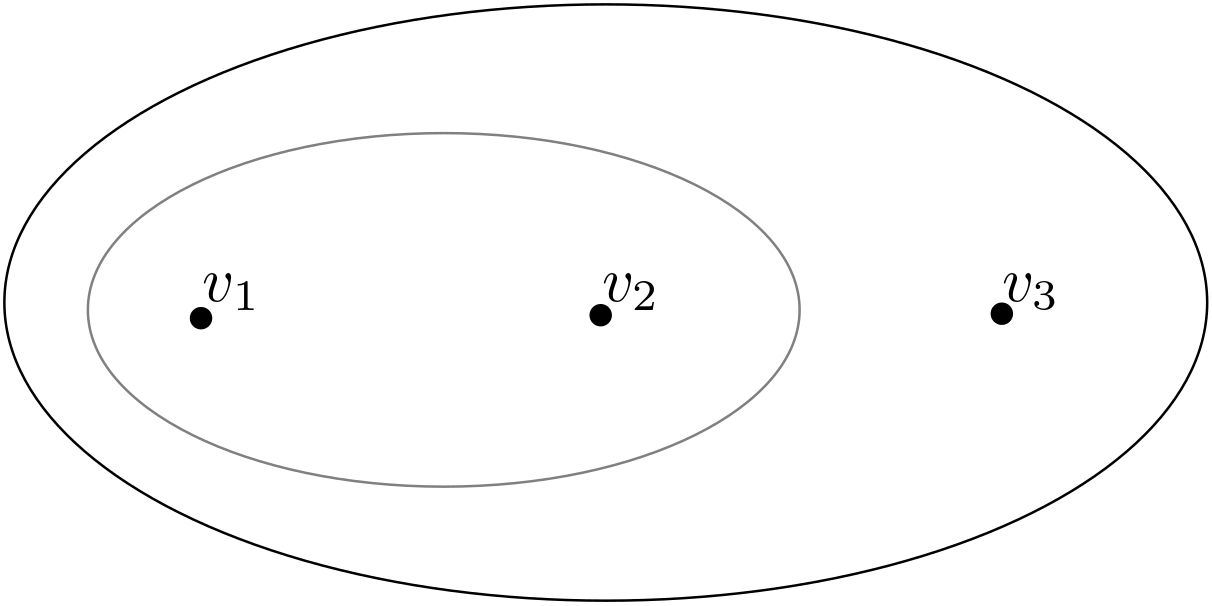}
    \caption{The hypergraph in Example \ref{ex:123}.}
    \label{fig:ex123}
\end{figure} In this case, $N=\nabla=3$. The non-zero entries of the adjacency tensor are
\begin{equation*}
    A_{112}=A_{121}=A_{211}=A_{122}=A_{212}=A_{221}=\frac{w(\{v_1,v_2\})}{N(2,3)}=\frac{1}{3}
\end{equation*}and
\begin{equation*}
    A_{123}=A_{132}=A_{213}=A_{231}=A_{312}=A_{321}=\frac{w(\{v_1,v_2,v_3\})}{N(3,3)}=1.
\end{equation*}Also, since $\deg v_1=\deg v_2=3$ and $\deg v_3=2$, the non-zero entries of $K$ are
\begin{align*}
    &  K_{112}=K_{121}=K_{211}=K_{122}=K_{212}=K_{221}=-\frac{1}{3},\\
    &  K_{123}=K_{132}=K_{213}=K_{231}=K_{312}=K_{321}=-1,\\
    & K_{111}=K_{222}=3,\qquad K_{333}=2.
\end{align*}Similarly, the non-zero entries of $\mathcal{L}$ are
\begin{align*}
    &  \mathcal{L}_{112}=\mathcal{L}_{121}=\mathcal{L}_{211}=\mathcal{L}_{122}=\mathcal{L}_{212}=\mathcal{L}_{221}=-\frac{1}{9},\\
    &  \mathcal{L}_{123}=\mathcal{L}_{132}=\mathcal{L}_{213}=\mathcal{L}_{231}=-\frac{1}{3},\\
    & \mathcal{L}_{312}=\mathcal{L}_{321}=-\frac{1}{2}\\
    & \mathcal{L}_{111}=\mathcal{L}_{222}=\mathcal{L}_{333}=1.
\end{align*}

\end{example}

\section{First properties}\label{sec: first properties}
In this section, we prove the first spectral properties of the hypergraph tensors that we introduced in the previous section. Some of our results generalize those in \cite[Chapter 4]{QL} to weighted hypergraphs.\newline
We start by proving that the normalized Laplacian of a weighted hypergraph has the same spectrum as the random walk Laplacian tensor,  as a generalization of Theorem 3.11 in \cite{Banerjee}, which is stated for the unweighted case.  The same holds for the signless versions. 
\begin{proposition}\label{prop:isospectral}
\begin{enumerate}
    \item The tensors $L$ and $\mathcal{L}$ have the same eigenvalues, counted with algebraic multiplicity.
\item $(\lambda,\mathbf{x})$ is an eigenpair for $L$ if and only if $(\lambda,\mathbf{y})$ is an eigepair for $\mathcal{L}$, where $\mathbf{y}\in\mathbb{C}^N$ is the vector with entries
\begin{equation*}
   y_j=\frac{x_j}{\sqrt[\nabla]{\deg v_{j}}}.
\end{equation*}Hence, the eigenvalues of $L$ and $\mathcal{L}$ have also the same geometric and span multiplicities, and $\hspec(L)=\hspec(\mathcal{L})$.\end{enumerate}
\end{proposition}The same holds for $L^+$ and $\mathcal{L}^+$.
\begin{proof}
We only prove the claims for $L$ and $\mathcal{L}$, as the other case is similar. 
\begin{enumerate}
    \item As recalled for instance in \cite[Page 2219]{fan2019spectral}  and in \cite[Theorem 2.3]{shao2013general}, if $d_1,\ldots,d_N\in \mathbb{R}$, then $L$ has the same eigenvalues, counted with algebraic multiplicity, as the $\nabla$-th order $N$-dimensional tensor $L'(d_1,\ldots,d_N)$ with entries
\begin{equation*}
    L'(d_1,\ldots,d_N)_{i_1,\ldots,i_\nabla}=\frac{1}{d_{i_1}^{\nabla-1}}\cdot L_{i_1,\ldots,i_\nabla}\cdot d_{i_2}\cdots d_{i_\nabla}.
\end{equation*}By taking $d_j=\sqrt[\nabla]{\deg v_j}$ for each $j\in[N]$, we have that $L'(d_1,\ldots,d_N)=\mathcal{L}$. This proves the claim for the algebraic multiplicities.
\item The claim for the other multiplicities follows directly from Theorem 2.5 in \cite{shao2013general}.
\end{enumerate}

\end{proof}

The following proposition generalizes Theorem 4.5(b) and (c) in \cite{QL}.

\begin{proposition}
Each of the tensors $A$, $K$, $K^+$, $L$, $L^+$, $\mathcal{L}$ and $\mathcal{L}^+$ has $(\nabla-1)^{N-1}\cdot N$ eigenvalues, counted with algebraic multiplicity, whose sum is:
   \begin{itemize}
    \item $0$, for $A$;
    \item $(\nabla-1)^{N-1}\cdot \left(\sum_{i=1}^N\deg v_i\right)$, for $K$ and $K^+$;
    \item $(\nabla-1)^{N-1}\cdot N$, for $L$, $L^+$, $\mathcal{L}$ and $\mathcal{L}^+$.
\end{itemize}
\end{proposition}
\begin{proof}
As shown in \cite[Section 3]{Qi}, a symmetric tensor has $(\nabla-1)^{N-1}\cdot N$ eigenvalues, counted with algebraic multiplicity, whose sum is $(\nabla-1)^{N-1}$ times the sum of its diagonal elements. Hence, the claim for $A$, $K$, $K^+$, $L$ and $L^+$ follows. The claim for $\mathcal{L}$ and $\mathcal{L}^+$ follows by Proposition \ref{prop:isospectral}.
\end{proof}

In Section \ref{section:hyp} we defined quite a few tensors. However, Proposition \ref{prop:isospectral} tells us that two of them are the same, from the spectrum viewpoint. The next two remarks are a further step in this direction. They show that the spectra of all our tensors are closely related.

\begin{remark}\label{rmk2-}
We have $L^+=2\cdot \id -L \quad \text{and} \quad  \mathcal{L}^+=2\cdot \id -\mathcal{L}$. Therefore, it is easy to check that
\begin{equation*}
    (\lambda,\mathbf{x}) \text{ is an eigenpair for }L \iff  (2-\lambda,\mathbf{x}) \text{ is an eigenpair for }L^+
\end{equation*}and $
    (\lambda,\mathbf{x}) \text{ is an eigenpair for }\mathcal{L} \iff  (2-\lambda,\mathbf{x}) \text{ is an eigenpair for }\mathcal{L}^+$.
Moreover,  since
\begin{equation*}
    K_{i_1,\ldots,i_\nabla}=(\deg v_{i_1})  \mathcal{L}_{i_1,\ldots,i_\nabla},
\end{equation*}we have that

\begin{equation*}
    (0,\mathbf{x}) \text{ is an eigenpair for }K \iff  (0,\mathbf{x}) \text{ is an eigenpair for }\mathcal{L}
\end{equation*}and similarly $(0,\mathbf{x}) \text{ is an eigenpair for }K^+ \iff  (0,\mathbf{x}) \text{ is an eigenpair for }\mathcal{L}^+$. Finally, $(0,\mathbf{x}) \text{ is an eigenpair for }A \iff  (1,\mathbf{x}) \text{ is an eigenpair for }\mathcal{L} \text{ and }\mathcal{L}^+$.
\end{remark}
\begin{remark}
If $G$ is $\Delta$-regular, then
\begin{equation*}
    K=\Delta\cdot \id-A,\quad L=\mathcal{L}=\frac{1}{\Delta}\cdot K\quad \text{and} \quad K^+=2\Delta\cdot \id-K.
\end{equation*}

Hence, in this case, it is easy to see that
\begin{align*}
    \lambda \text{ is an eigenvalue for }K&\iff \Delta-\lambda \text{ is an eigenvalue for }A\\
    &\iff \frac{\lambda}{\Delta} \text{ is an eigenvalue for }L=\mathcal{L}\\
     &\iff 2-\frac{\lambda}{\Delta} \text{ is an eigenvalue for }L^+=\mathcal{L}^+\\
    &\iff 2\Delta-\lambda \text{ is an eigenvalue for }K^+,
\end{align*}with the same multiplicities. In particular, the spectral theories of the different tensors are equivalent to each other for regular weighted hypergraphs.
\end{remark}

Such observations allow us to expand on the previous knowledge on their eigenvalues. For instance, we are in position to discuss existence of H-eigenvectors and H-eigenvalues.
\begin{proposition}\label{prop:Heigenvalues}
\begin{enumerate}
    \item The tensors $A$, $K^+$, $L^+$ and $\mathcal{L}^+$ have at least one $H$-eigenvalue. Their largest $H$-eigenvalue equals their spectral radius, and has a non-negative $H$-eigenvector.
    \item The tensors $L$ and $\mathcal{L}$ have at least one $H$-eigenvalue. Their smallest $H$-eigenvalue equals $2-\rho(L)$, and has a non-negative $H$-eigenvector.
\end{enumerate}
\end{proposition}
\begin{proof}
The first claim follows from \cite[Theorem 2.4]{QL}, which applies to non-negative tensors. The second claim follows from the first one, together with Proposition \ref{prop:isospectral} and Remark \ref{rmk2-}. %The third claim follows from Theorem 5 in \cite{Qi}, which applies to symmetric tensors.
\end{proof}

Now we want to prove that our hypergraph  Laplacian  tensors are diagonally dominated. For this purpose we compute the sums of their rows, which is an interesting result in itself. This will also allow to bound their H-eigenvalues.

\begin{lemma}\label{lemma:rowsums}
Given $i_1\in [N]$,
\begin{align*}
 &\sum_{i_2,\ldots,i_\nabla\in [N]} A_{i_1,i_2,\ldots,i_\nabla}=\deg v_{i_1},\\
 &\sum_{i_2,\ldots,i_\nabla\in [N]} K_{i_1,i_2,\ldots,i_\nabla}=\sum_{i_2,\ldots,i_\nabla\in [N]} \mathcal{L}_{i_1,i_2,\ldots,i_\nabla}=0,\\
 &\sum_{i_2,\ldots,i_\nabla\in [N]} K^+_{i_1,i_2,\ldots,i_\nabla}=2\cdot \deg v_{i_1},\\
 &\sum_{i_2,\ldots,i_\nabla\in [N]} \mathcal{L}^+_{i_1,i_2,\ldots,i_\nabla}=2.
\end{align*}

\end{lemma}
\begin{proof}
Given $i_1\in [N]$,
\begin{align*}
   \sum_{i_2,\ldots,i_\nabla\in [N]} A_{i_1,i_2,\ldots,i_\nabla}&= \sum_{r=2}^\nabla\left(\sum_{e\in E_r:\, v_{i_1}\in e}|\text{entries of the row $A_{i_1}$ corresponding to }e|\cdot \frac{w(e)}{N(r,\nabla)}\right)\\
   &=\sum_{r=2}^\nabla\, \sum_{e\in E_r:\, v_{i_1}\in e} w(e)=\deg v_{i_1}.
\end{align*}This implies that 
\begin{equation*}
   \sum_{i_2,\ldots,i_\nabla\in [N]} K_{i_1,i_2,\ldots,i_\nabla} =\deg v_{i_1}- \sum_{i_2,\ldots,i_\nabla\in [N]} A_{i_1,i_2,\ldots,i_\nabla}=0,
\end{equation*}while
\begin{equation*}
    \sum_{i_2,\ldots,i_\nabla\in [N]} \mathcal{L}_{i_1,i_2,\ldots,i_\nabla}=\frac{1}{\deg v_{i_1}}\left(\sum_{i_2,\ldots,i_\nabla\in [N]} K_{i_1,i_2,\ldots,i_\nabla}\right)=0.
\end{equation*}Similarly,
\begin{equation*}
    \sum_{i_2,\ldots,i_\nabla\in [N]} K^+_{i_1,i_2,\ldots,i_\nabla}=\deg v_{i_1}+ \sum_{i_2,\ldots,i_\nabla\in [N]} A_{i_1,i_2,\ldots,i_\nabla}=2\cdot \deg v_{i_1}
\end{equation*}and
\begin{equation*}
    \sum_{i_2,\ldots,i_\nabla\in [N]} \mathcal{L}^+_{i_1,i_2,\ldots,i_\nabla}=\frac{1}{\deg v_{i_1}}\left(\sum_{i_2,\ldots,i_\nabla\in [N]} K^+_{i_1,i_2,\ldots,i_\nabla}\right)=\frac{1}{\deg v_{i_1}}\left(2\cdot \deg v_{i_1}\right)=2.\qedhere
\end{equation*}
\end{proof}

Now we show that hypergraph  Laplacian  tensors are diagonally dominated,  and we prove some bounds for the eigenvalues of all hypergraph tensors eigenvalues which generalize Theorem 3.1 in \cite{Banerjee}, Theorem 4.5(e) in \cite{QL} and Theorem 3.13(ii) in \cite{Banerjee}.
\begin{theorem}
\begin{enumerate}
    \item The tensors $K$, $\mathcal{L}$, $K^+$ and $\mathcal{L}^+$ are diagonally dominated.
 \item If $\lambda$ is an eigenvalue for $K$ or $K^+$, then 
\begin{equation*}
    |\lambda-\Delta|\leq \Delta.
\end{equation*}
If $\mu$ is an eigenvalue for $\mathcal{L}$ or $\mathcal{L}^+$, then
\begin{equation*}
    |\mu-1|\leq 1.
\end{equation*}
If $\nu$ is an eigenvalue for $A$, then
\begin{equation*}
    |\nu|\leq \Delta.
\end{equation*}
In particular, all the $H$-eigenvalues of $K$ and $K^+$ are in  $[0,2\Delta]$,  all the $H$-eigenvalues of $\mathcal{L}$ and $\mathcal{L}^+$ (equivalently, $L$ and $L^+$) are in $[0,2]$, and all the $H$-eigenvalues of $A$ are in $[-\Delta,\Delta]$. \end{enumerate}
\end{theorem}
\begin{proof}\begin{enumerate}
    \item By Lemma \ref{lemma:rowsums}, for each $i\in[N]$ we have
\begin{align*}
    & K_{i,\ldots,i}=\deg v_i= \sum_{\substack{i_2,\ldots,i_\nabla\in [N]\\ \text{not all equal to }i}}\left|K_{i,i_2,\ldots,i_\nabla}\right|;\\
    & K^+_{i,\ldots,i}=\deg v_i= \sum_{\substack{i_2,\ldots,i_\nabla\in [N]\\ \text{not all equal to }i}}\left|K^+_{i,i_2,\ldots,i_\nabla}\right|;\\
     & \mathcal{L}_{i,\ldots,i}=1= \sum_{\substack{i_2,\ldots,i_\nabla\in [N]\\ \text{not all equal to }i}}\left|\mathcal{L}_{i,i_2,\ldots,i_\nabla}\right|;\\
    & \mathcal{L}^+_{i,\ldots,i}=1= \sum_{\substack{i_2,\ldots,i_\nabla\in [N]\\ \text{not all equal to }i}}\left|\mathcal{L}^+_{i,i_2,\ldots,i_\nabla}\right|.
\end{align*}Hence, $K$, $\mathcal{L}$, $K^+$ and $\mathcal{L}^+$ satisfy \eqref{eq:def-dd}, implying that they are diagonally dominated. 
\item  The second claim follows from Equation \eqref{eq:dd}.\end{enumerate}
\end{proof}

%\begin{remark}For any $\nabla$, the Kirchhoff Laplacian $K$ is a \emph{copositive tensor} (cf.\ \cite{Qi_Symm}).\end{remark}

Now we move to irreducibility. Generalizing \cite[Theorem 4.1]{QL}, we give a necessary and sufficient condition for our tensors to be weakly irreducible, as defined in Section \ref{sec: intro}.
\begin{theorem}\label{thm:weaklyirr}
$A$, $K$, $K^+$, $L$, $L^+$, $\mathcal{L}$ and $\mathcal{L}^+$ are weakly irreducible tensors if and only if $G$ is connected.
\end{theorem}
\begin{proof}
Without loss of generality, we only prove the claim for $A$. By definition, the tensor $A$ is weakly irreducible if and only if for any non-empty proper index subset $J$ of $[N]$, there is at least one entry
\begin{equation*}
    A_{i_1,\ldots,i_\nabla}\neq 0,
\end{equation*}where $i_1\in J$ and at least one index $i_j\in [N]\setminus J$, for $j=2,\ldots,\nabla$. By definition of $A$, this happens if and only if, for each non-empty proper subset $J$ of $[N]$, there exist $i_1\in J$ and $i_j\in [N]\setminus J$ such that $v_{i_1}$ and $v_{i_j}$ share a common edge. Hence, $A$ is weakly irreducible if and only if the hypergraph is connected.\end{proof}

Now that weak irreducibility is settled, we address irreducibility.  We introduce a new  combinatorial property of the hypergraph that will allow us to characterize irreducible hypergraph tensors.

\begin{definition}\label{def:red}
A hypergraph $G=(V,E,w)$ is \emph{reducible} if one can decompose the vertex set as a disjoint union $V=V_1\sqcup V_2$ such that $V_1$ and $V_2$ are both non-empty and, for each edge $e$, 
\begin{equation*}
  e\cap V_1\neq \emptyset\Rightarrow  |e\cap V_1|\geq 2.
\end{equation*}
A hypergraph is \emph{irreducible} if it is not reducible.
\end{definition}
\begin{figure}[ht]
    \centering
    \includegraphics[width=8cm]{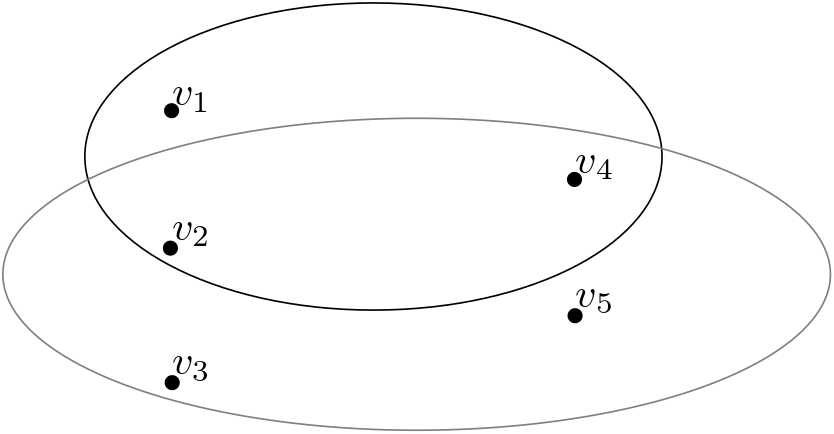}
    \caption{A reducible hypergraph. This can be seen by taking $V_1=\{v_1,v_2,v_3\}$ and $V_2=\{v_4,v_5\}$.}
    \label{figure:red}
\end{figure}

Let us consider some examples:
\begin{itemize}
    \item Every disconnected hypergraph is reducible.
    \item If $G$ has one vertex $v$ that is not contained in any edge of cardinality $2$, then, by setting $V_1=V\setminus \{v\}$ and $V_2=\{v\}$, it is clear that $G$ is reducible. This implies, in particular, that the majority of connected hypergraphs are reducible.
   \item If $G=(V,E)$ and there exists $E'\subseteq E$ such that $G'=(V,E')$ is a connected graph, then $G$ is irreducible. In particular, every connected graph is irreducible.
\end{itemize}

\begin{theorem}
$A$, $K$, $K^+$, $L$, $L^+$, $\mathcal{L}$ and $\mathcal{L}^+$ are irreducible tensors if and only if $G$ is an irreducible hypergraph.
\end{theorem}
\begin{proof}
If $G$ is reducible, let $V=V_1\sqcup V_2$ be a decomposition of the vertex set as in Definition \ref{def:red}. Then, setting $J=\{i\in [N]:v_i\in V_1\}$ shows that the tensors associated to $G$ are reducible. Vice versa, if $J\subset [N]$ shows that the tensors associated to $G$ are reducible, then setting $V_1=\{v_i\in V:i\in J\}$ and $V_2=V\setminus V_1$ shows that $G$ is a reducible hypergraph.
\end{proof}

Another important property that we want to understand better is the spectral radius.  As a generalization of \cite[Theorem 4.4]{QL}, which is stated for the unweighted, $\nabla$-uniform case,  we are able to give tight bounds on $\rho(A)$ and $\rho(K^+)$ and to compute $\rho(L^+)$ and $\rho(\LL^+)$ for a weighted hypergraph.

\begin{proposition}The spectral radii of $A$ and $K^+$ satisfy
\begin{equation*}
  \delta\leq \rho(A)\leq \Delta\mbox{ and }
  2\delta \leq \rho(K^+)\leq 2 \Delta.
\end{equation*}
In particular, if $G$ is $\Delta$-regular, then
\begin{equation*}
    \rho(A)= \Delta \text{ and } \rho(K^+)= 2\Delta.
\end{equation*}
Moreover, if $G'=(V,E',w')$ is another weighted hypergraph on $N$ vertices and maximum edge cardinality $\nabla$, and it is obtained from $G$ by removing edges or by decreasing some edge weights, then
\begin{equation*}
    \rho(A(G'))\leq \rho(A(G))  \text{ and } \rho(K(G'))\leq \rho(K(G)).
\end{equation*}
\end{proposition}
\begin{proof}
If $T$ is a non-negative tensor, then by \cite[Lemma 3.20]{QL},
\begin{equation*}
   \min_{i_1\in [N]}\sum_{i_2,\ldots,i_\nabla\in [N]} T_{i_1,i_2,\ldots,i_\nabla} \leq \rho(T)\leq \max_{i_1\in [N]}\sum_{i_2,\ldots,i_\nabla\in [N]} T_{i_1,i_2,\ldots,i_\nabla}.
\end{equation*}
By Lemma \ref{lemma:rowsums}, this implies that 
  $\delta\leq \rho(A)\leq \Delta$ and
 $2\delta \leq \rho(K^+)\leq 2 \Delta$. Moreover, if $G'=(V,E',w')$ is another weighted hypergraph on $N$ vertices and maximum edge cardinality $\nabla$, and it is obtained from $G$ by removing edges or by decreasing some edge weights, then
\begin{equation*}
    A(G')_{i_1,\ldots,i_\nabla}\leq A_{i_1,\ldots,i_\nabla}  \text{ and } K^+(G')_{i_1,\ldots,i_\nabla}\leq K^+_{i_1,\ldots,i_\nabla},
\end{equation*}for each $i_j\in [N]$ and $j\in [\nabla]$. By \cite[Lemma 2.3]{rho}, this implies that
\begin{equation*}
    \rho(A(G'))\leq \rho(A)  \text{ and } \rho(K(G'))\leq \rho(K).\qedhere
\end{equation*}
\end{proof}

\begin{proposition}\label{prop_radiusL}
The spectral radius of $L^+$ and $\mathcal{L}^+$ is $\rho(L^+)=\rho(\mathcal{L}^+)=2$.
\end{proposition}
\begin{proof}
Since $L^+$ and $\mathcal{L}^+$ are isospectral, they have the same spectral radius. Now, since $\mathcal{L}$ is a non-negative tensor, by \cite[Lemma 3.20]{QL},
\begin{equation*}
   \min_{i_1\in [N]}\sum_{i_2,\ldots,i_\nabla\in [N]} \mathcal{L}^+_{i_1,i_2,\ldots,i_\nabla} \leq \rho(\mathcal{L}^+)\leq \max_{i_1\in [N]}\sum_{i_2,\ldots,i_\nabla\in [N]} \mathcal{L}^+_{i_1,i_2,\ldots,i_\nabla}.
\end{equation*}By Lemma \ref{lemma:rowsums}, the claim follows.
\end{proof}

Thanks to Proposition \ref{prop_radiusL}, we can derive more information on the H-eigenvalues and H-eigenvectors.

\begin{corollary}\label{cor:2eigenvalue}
If $G$ is connected, then $2$ is an eigenvalue for $L^+$ and $\mathcal{L}^+$, and $(1,\ldots,1)\in \mathbb{R}^N$ is the unique positive $H$-eigenvector of $\mathcal{L}^+$ associated to $2$, up to a multiplicative constant.
\end{corollary}
\begin{proof}
The fact that $2$ is always an eigenvalue for $L^+$ and $\mathcal{L}^+$ follows by Proposition \ref{prop:Heigenvalues} and Proposition \ref{prop_radiusL}. Moreover, it is easy to check that a corresponding eigenvector for $\mathcal{L}^+$ is $(1,\ldots,1)\in \mathbb{R}^N$. By Theorem \ref{minmax:waklyirr}, the claim follows.
\end{proof}

The following corollary generalizes Theorems 3.9(i) and 3.13(iv) in \cite{Banerjee}.

\begin{corollary}
If $G$ is connected, then $0$ is an eigenvalue for $K$, $L$ and $\mathcal{L}$, and $(1,\ldots,1)\in \mathbb{R}^N$ is the unique positive $H$-eigenvector of $\mathcal{L}$ associated to $0$, up to a multiplicative constant.
\end{corollary}
\begin{proof}
It follows from Remark \ref{rmk2-} and Corollary \ref{cor:2eigenvalue}.
\end{proof}

The last two results concerned connected hypergraphs. From our viewpoint, it is not very restrictive to assume that the hypergraph is indeed connected. If this is not the case, we can study the spectrum of a hypergraph from the spectra of its connected components. The following theorem generalizes \cite[Theorem 3.15]{Banerjee}.
\begin{theorem}
Given two hypergraphs $G_1=(V_1,E_1,w_1)$ and $G_2=(V_2,E_2,w_2)$, let $G:=G_1\sqcup G_2=(V_1\sqcup V_2, E_1\sqcup E_2,w)$, where $w|_{E_i}:=w_i$ for $i=1,2$. Let also $T\in \{A,K,K^+,L,L^+,\mathcal{L},\mathcal{L}^+\}$. Then the eigenvalues of $T(G)$ are precisely the eigenvalues of $T(G_1)$ together with the eigenvalues of $T(G_2)$. Moreover, an eigenvalue that has algebraic multiplicity $m$ for $T(G_1)$ has algebraic multiplicity $m(\nabla-1)^{|V_2|}$ for $T(G)$.
\end{theorem}
\begin{proof}
This is an immediate consequence of \cite[Corollary 4.2]{Shao}.
\end{proof}

\section{Duplicate vertices}\label{section: duplicate}
In this section we discuss \emph{duplicate vertices}, i.e.\ vertices that do not share common edges but are structurally equivalent. It is known that, in the case of graphs, duplicate vertices leave a signature in the spectra of the operators. Also, such vertices are important in applied network theory because their presence is important for the study of the \emph{network redundancy} and \emph{robustness}, as discussed in \cite{symm}.

\begin{definition}
Two vertices $v_i$ and $v_j$ are \emph{duplicate} if they do not share common edges and the corresponding rows of the adjacency tensor are the same, that is,
\begin{equation*}
    A_{i,i_2,\ldots,i_\nabla}=A_{j,i_2,\ldots,i_\nabla}\text{ for every  } i_2,\ldots,i_\nabla\in[N].
\end{equation*}
\end{definition}
Note that the above definition of duplicate vertices, that we first introduced here,  does not coincide with \cite[Definition 3.2]{MulasZhang}, which is based on the hypergraph adjacency matrix. However, both these definitions coincide with the classical one in the case of simple graphs.  It follows from the definition that if $v_i$ is a vertex which has a duplicate $v_j$, then $v_i$ is not contained in any edge of cardinality smaller than $\nabla$. Indeed, if $\nabla=2$, this is trivially true. If $\nabla\geq 3$, the existence of such an edge would imply that $A_{i,i,i_3,\dots,i_{\nabla}}\neq 0$, for some $\{i_3,\dots,i_{\nabla}\}\subset [N]$. But as $v_j$ is a duplicate, then $A_{j,i,i_3,\dots,i_{\nabla}}\neq 0$, which contradicts the fact that $v_i$ and $v_j$ are not contained in a common edge.\newline
The following theorem and proposition are well known for simple graphs, but they are new for general hypergraphs also in the unweighted, uniform case.

\begin{theorem}\label{thm:duplicate}
Assume that there are $n$ vertices which are duplicate of each other.
\begin{enumerate}
\item If $\nabla=2$, then
\begin{itemize}
    \item $0$ is an eigenvalue for $A$ with multiplicity at least $n-1$;
    \item $1$ is an eigenvalue for $L$, $L^+$, $\mathcal{L}$ and $\mathcal{L}^+$, with multiplicity at least $n-1$. 
\end{itemize}
\item If $\nabla\ge 3$, then
\begin{itemize}
    \item $0$ is an eigenvalue for $A$ with Hspan multiplicity at least $n$ and geometric multiplicity at least $n$;
    \item $1$ is an eigenvalue for $L$, $L^+$, $\mathcal{L}$ and $\mathcal{L}^+$, with Hspan multiplicity at least $n$ and geometric multiplicity at least $n$.
\end{itemize}
\end{enumerate}
\end{theorem}
\begin{proof}We only prove the claims for $A$, as the other claims then follow from Remark \ref{rmk2-}.  Assume that $v_1,\ldots,v_n$ are duplicate vertices.\newline
If $\nabla = 2$, then $A$ is a symmetric matrix. Therefore, as recalled in the introduction, each eigenvalue is an $H$-eigenvalue, and all multiplicities coincide. We observe that any non-zero vector $\mathbf{x}\in \mathbb{R}^N$ satisfying $x_i=0$ for $i>n$ and $\sum_{i=1}^n x_i =0$ is an eigenvector with eigenvalue $0$. In particular, the $n-1$ linearly independent vectors $e_1-e_j$, for $2\leq j\leq n$, where $e_i$ is the $i$-th vector of the canonical basis of $\mathbb{R}^N$, are eigenvectors of $0$ for $A$.\newline
If $\nabla\ge 3$, we have already observed that $v_i$ is not contained in any edge of cardinality smaller than $\nabla$. This implies that every monomial in the equations $(A\mathbf{x}^{\nabla-1})_{i}$ with $i>n$ is divisible by at least one variable $x_j$, with $j>n$. %It follows that any non-zero vector $\mathbf{x}\in \mathbb{C}^N$ satisfying $x_i=0$ for $i>n$ is an eigenvector with eigenvalue $0$. 
This means that the eigenvariety contains the $n$-dimensional linear space defined by the equations
\[x_{n+1}=\ldots=x_N=0,
\]
therefore $\gm(0)\ge n$.
In particular, $e_1,\ldots,e_n\in \mathbb{R}^N$ are eigenvectors with eigenvalue $0$, so $\Hsm(0)\ge n$.
\end{proof}

\begin{proposition}\label{prop:duplicate}
Let $v_i$ and $v_j$ be duplicate vertices.
\begin{itemize}
    \item If $(\lambda,\mathbf{x})$ is an eigenpair for $A$ and $\lambda\neq 0$, then $x_i^{\nabla-1}=x_j^{\nabla-1}$.
    \item If $(\mu,\mathbf{y})$ is an eigenpair for $L$, $L^+$, $\mathcal{L}$ or $\mathcal{L}^+$ and $\mu\neq 1$, then $y_i^{\nabla-1}=y_j^{\nabla-1}$.
    \item  If $(\nu,\mathbf{z})$ is an eigenpair for $K$ or $K^+$ and $\nu\neq \deg v_i$, then $z_i^{\nabla-1}=z_j^{\nabla-1}$. 
\end{itemize}
\end{proposition}
\begin{proof}
We only prove the claim for $A$, the other cases being similar. Since $(\lambda,\mathbf{x})$ is an eigenpair for $A$,
\begin{equation}\label{eq:v_i}
    \sum_{i_2,\ldots,i_\nabla\in [N]} A_{i,i_2,\ldots,i_\nabla}x_{i_2}\cdots x_{i_\nabla}=\lambda x_i^{\nabla-1}
\end{equation}and
\begin{equation}\label{eq:v_j}
    \sum_{i_2,\ldots,i_\nabla\in [N]} A_{j,i_2,\ldots,i_\nabla}x_{i_2}\cdots x_{i_\nabla}=\lambda x_j^{\nabla-1}.
\end{equation}Since $v_i$ and $v_j$ are duplicate, the left-hand sides in \eqref{eq:v_i} and \eqref{eq:v_j} coincide. Thus,
\begin{equation*}
    \lambda x_i^{\nabla-1}=\lambda x_j^{\nabla-1}.
\end{equation*}Since $\lambda\neq 0$, this implies that $x_i^{\nabla-1}=x_j^{\nabla-1}$.
\end{proof}

\section{The hyperflower}\label{section: iperfiore}

The $\nabla$-\emph{hyperflower} %\cite{AndreottiMulas}
is the unweighted, $\nabla$-uniform hypergraph $G=(V,E)$ on $N$ nodes and $M=N-\nabla+1$  edges, such that% (Figures \ref{fig:hyperflower} and \ref{fig:flower2}):
\begin{itemize}
    \item $V=\{v_1,\ldots,v_N\}$
    \item $E=\{\ell_{\nabla},\ldots,\ell_{N}\}$
    \item $\ell_j=\{v_1,\ldots,v_{\nabla-1},v_j\}$ for every $j\in\{\nabla,\ldots,N\}$.
\end{itemize}
We say that the vertices $v_1,\ldots,v_{\nabla-1}$ are the \emph{central vertices} of $G$, while $v_{\nabla},\ldots,v_N$ are its \emph{peripheral vertices.} \newline

If $\nabla=2$, then the $\nabla$-hyperflower is simply the star graph, and its spectrum with respect to any operator is well known. For $\nabla\geq 3$, the results in this section are new. \newline  

If $G$ is a $\nabla$-hyperflower, all its central vertices belong to all edges, hence they all have degree $M$. Moreover, the $M$ peripheral vertices of $G$ have degree $1$. They are duplicate of each other and therefore, by Theorem \ref{thm:duplicate}, $0$ is an eigenvalue for $A$ with Hspan multiplicity at least $M$, while $1$ is an eigenvalue for $L$, $L^+$, $\mathcal{L}$ and $\mathcal{L}^+$, with Hspan multiplicity at least $M$. In the following proposition we improve this result for hypergraphs that have no edges of size $2$, as for instance the $\nabla$-hyperflower for $\nabla\geq 3$.

\begin{proposition}\label{prop:0flower}
Let $G$ be an hypergraph with $N$ vertices such that every edge contains at least three vertices. Then
\begin{itemize}
    \item $0$ is an eigenvalue for $A$, with Hspan multiplicity equal to $N$, and
    \item $1$ is an eigenvalue for $L$, $L^+$, $\mathcal{L}$ and $\mathcal{L}^+$, with Hspan multiplicity equal to $N$.
\end{itemize}
In particular, this holds for the $\nabla$-hyperflower whenever $\nabla\ge 3$.\end{proposition}

\begin{proof}
We only prove the claim for $A$. The other claims then follow from Remark \ref{rmk2-}. By hypothesis, an entry $A_{i_1,\ldots, i_{\nabla}}$ of $A$ is nonzero only if $|\{i_1,\ldots, i_{\nabla}\}|\ge 3$. Therefore, given $i\in[N]$, each monomial of the polynomial
\[\sum_{i_2,\ldots,i_k\in [N]} A_{i,i_2,\ldots,i_\nabla}x_{i_2}\cdots x_{i_\nabla}\]
involves at least two different variables. This implies that such polynomial vanishes on every vector of the canonical basis of $\R^N$, thus $0\in\spec(A)$ and $\textrm{Hsm}(0)=N$.
\end{proof}

\begin{theorem}\label{thm:Ahyperflower}
Let $G$ be a $\nabla$-hyperflower. Given $\lambda\in \mathbb{C}\setminus\{0\}$ and $\mathbf{x}\in \mathbb{C}^N\setminus \{\mathbf{0}\}$, we have that $(\lambda,\mathbf{x})$ is an eigenpair for $A$ if and only if, up to multiplying $\mathbf{x}$ by a non-zero constant, all the following conditions hold:
\begin{enumerate}
    \item $x_j^{\nabla-1}=1$ for each $j\in\{\nabla,\ldots,N\}$;
    \item $x_i^{\nabla}=\sum_{j=\nabla}^N x_j$
for each $i\in [\nabla-1]$;
    \item $\lambda=x_1\cdots x_{\nabla-1}$.
\end{enumerate}
\end{theorem}
\begin{proof}By Proposition \ref{prop:duplicate}, if $(\lambda,\mathbf{x})$ is an eigenpair for $A$ and $\lambda\neq 0$, then $x_j^{\nabla-1}$ must be constant for $j\in\{\nabla,\ldots,N\}$.  Assume first that $x_j=0$ for all $j\in\{\nabla,\ldots,N\}$. Then, for $i\in \{1,\ldots,\nabla-1\}$, the left side of
\begin{equation*}
    \sum_{i_2,\ldots,i_\nabla\in [N]} A_{i,i_2,\ldots,i_\nabla}x_{i_2}\cdots x_{i_\nabla}=\lambda x_i^{\nabla-1}
\end{equation*}vanishes. This implies that $\lambda x_i^{\nabla-1}=0$, therefore $x_i=0$ for each $i\in [\nabla-1]$, but this is a contradiction since $\mathbf{x}\neq \mathbf{0}$. Hence, $x_j^{\nabla-1}$ is constant and non-zero for $j\in\{\nabla,\ldots,N\}$.  Up to multiplying $\mathbf{x}$ by a non-zero constant, we can assume that  
\begin{equation}\label{eq:root1}
  x_j^{\nabla-1}=1\text{ for each }j\in\{\nabla,\ldots,N\}.
\end{equation}
Now, since $G$ is $\nabla$-uniform, $(\lambda,\mathbf{x})$ is an eigenpair for $A$ if and only if 
\begin{equation}\label{eq:uniform}
    \sum_{(v_k,v_{i_2},\ldots,v_{i_\nabla})\in E}x_{i_2}\cdots x_{i_\nabla} =\lambda x_k^{\nabla-1}
\end{equation}
for each $k\in [N]$. If $v_j$ is a peripheral vertex, then by \eqref{eq:root1}, \eqref{eq:uniform} becomes
\begin{equation}\label{eq:productlambda}
    x_1\cdots x_{\nabla-1} =\lambda.
\end{equation}Hence, since $\lambda\neq 0$, we have that $x_i\neq 0$ for each $i\in [\nabla-1]$.\newline
If $v_i$ is a central vertex, i.e.\ $i\in [\nabla-1]$, then \eqref{eq:uniform} becomes
\begin{equation*}
        \sum_{j=\nabla}^N x_j\cdot\frac{x_1\cdots x_{\nabla-1}}{x_i} =\lambda x_i^{\nabla-1}.
\end{equation*}By \eqref{eq:productlambda}, the latter equality is equivalent to
\begin{equation}\label{eq:nablaji}
    \sum_{j=\nabla}^N x_j=x_i^{\nabla},
\end{equation}for each $i\in [\nabla-1]$. This proves the claim.
\end{proof}

\begin{remark}
If $(\lambda,\mathbf{x})$ satisfies Theorem \ref{thm:Ahyperflower}, then $x_i^\nabla$ is constant for all $i\in [\nabla-1]$ and
\begin{equation}\label{eq:radiceM}
 \left|x_i^{\nabla}\right|=\left| \sum_{j=\nabla}^N x_j\right|\leq \sum_{j=\nabla}^N \left| x_j\right|=M,
\end{equation}implying that $|x_i|\leq \sqrt[\nabla]{M}$. Hence,
\begin{equation*}
    |\lambda|= \left|x_1\cdots x_{\nabla-1}\right|\leq \sqrt[\nabla]{M^{\nabla-1}}.
\end{equation*}
\end{remark}

\begin{corollary}\label{cor:hyperflower}
If $G$ is a $\nabla$-hyperflower with $M$ edges and $\omega$ is a $\nabla$-th root of $1$, then
    \begin{equation*}
        \sqrt[\nabla]{M^{\nabla-1}},\quad \omega\sqrt[\nabla]{M^{\nabla-1}},\quad \ldots,\quad \omega^{\nabla-1}\sqrt[\nabla]{M^{\nabla-1}}
    \end{equation*}are eigenvalues of $A(G)$. If, furthermore, $M=n(\nabla-1)+1$ for some positive integer $n$, then $1, \omega, \ldots, \omega^{\nabla-1}$ are also eigenvalues of $A(G)$. 
\end{corollary}

\begin{proof}
If $\omega$ is a $\nabla$-th root of $1$ and $\beta$ is a $(\nabla-1)$-th root of $\omega$, let
\begin{itemize}
    \item $x_j:=\beta^\nabla$ for each $j\in\{\nabla,\ldots,N\}$, and 
    \item $x_i:=\beta\sqrt[\nabla]{M}$ for each $i\in[\nabla-1]$.
\end{itemize}By Theorem \ref{thm:Ahyperflower}, $\mathbf{x}$ is an eigenvector for the eigenvalue
    \begin{equation*}
x_i^{\nabla-1}=\omega\sqrt[\nabla]{M^{\nabla-1}}.
    \end{equation*}Hence
    \begin{equation*}
     \sqrt[\nabla]{M^{\nabla-1}},\quad \omega\sqrt[\nabla]{M^{\nabla-1}},\quad \ldots,\quad \omega^{\nabla-1}\sqrt[\nabla]{M^{\nabla-1}}
    \end{equation*}are eigenvalues of $G$.\newline
 Now, assume that $M=n(\nabla-1)+1$ for some positive integer $n$. Let again $\omega$ be a $\nabla$-th root of $1$. Let $\alpha$ be a $(\nabla-1)$-th root of $\omega$ and let $z:=\alpha^\nabla$, so that $z$ is a $(\nabla-1)$-th root of $1$. Assume that the $M$ elements $x_{\nabla},\ldots,x_N$ are given by
    \begin{equation*}
        z \text{ ($n+1$ times) }, \quad z^2\text{ ($n$ times) }, \quad \ldots, \quad 1 \text{ ($n$ times)}.
    \end{equation*}Then, since $\sum_{k=1}^{\nabla-1}z^k=0$, by Theorem \ref{thm:Ahyperflower} we must have
    \begin{equation*}
    \sum_{j=\nabla}^N x_j=z=x_i^{\nabla},
\end{equation*}if we want $\mathbf{x}$ to be an eigenvector. If, in particular, $x_i:=\alpha$ for each $i\in [\nabla-1]$, then the above condition is satisfied and, by Theorem \ref{thm:Ahyperflower}, $\mathbf{x}$ is an eigenvector with eigenvalue
\begin{equation*}
x_i^{\nabla-1}=\alpha^{\nabla-1}=\omega.
\end{equation*}Hence    $  1, \omega, \ldots, \omega^{\nabla-1}$ are eigenvalues of $A(G)$. 
\end{proof}

For simplicity, in Theorem \ref{thm:Ahyperflower} and Corollary \ref{cor:hyperflower}, we focused on the adjacency tensor of a $\nabla$-hyperflower $G$. However, similar results can be shown also for the other tensors associated to $G$. The following two theorems for $K(G)$ and $\mathcal{L}(G)$, respectively, can be proved as Theorem \ref{thm:Ahyperflower}.

\begin{theorem}
Let $G$ be a $\nabla$-hyperflower with $M$ edges. Given $\lambda\in \mathbb{C}\setminus\{1,M\}$ and $\mathbf{x}\in \mathbb{C}^N\setminus \{\mathbf{0}\}$, we have that $(\lambda,\mathbf{x})$ is an eigenpair for $K$ if and only if, up to multiplying $\mathbf{x}$ by a non-zero constant, all the following conditions hold:
\begin{enumerate}
    \item $x_j^{\nabla-1}=1$ for each $j\in\{\nabla,\ldots,N\}$;
    \item $x_i^{\nabla}=\frac{(1-\lambda)}{(M-\lambda)}\cdot \left(\sum_{j=\nabla}^N x_j\right)$
for each $i\in [\nabla-1]$;
    \item $\lambda=1-x_1\cdots x_{\nabla-1}$.
\end{enumerate}
\end{theorem}

\begin{theorem}
Let $G$ be a $\nabla$-hyperflower with $M$ edges. Given $\lambda\in \mathbb{C}\setminus\{1\}$ and $\mathbf{x}\in \mathbb{C}^N\setminus \{\mathbf{0}\}$, we have that $(\lambda,\mathbf{x})$ is an eigenpair for $\mathcal{L}$ if and only if, up to multiplying $\mathbf{x}$ by a non-zero constant, all the following conditions hold:
\begin{enumerate}
    \item $x_j^{\nabla-1}=1$ for each $j\in\{\nabla,\ldots,N\}$;
    \item $x_i^{\nabla}=\frac{1}{M}\cdot \left(\sum_{j=\nabla}^N x_j\right)$
for each $i\in [\nabla-1]$;
    \item $\lambda=1-x_1\cdots x_{\nabla-1}$.
\end{enumerate}
\end{theorem}

    \begin{example}\label{ex:hyperflower}
Let $G=(V,E)$ be the $3$-hyperflower with $V=\{v_1,v_2,v_3,v_4\}$ and $E = \{\{v_1,v_2,v_3\},\{v_1,v_2,v_4\}\}$ (Figure \ref{fig:hyperflower}).
\begin{figure}[ht]
    \centering
    \includegraphics[width=7cm]{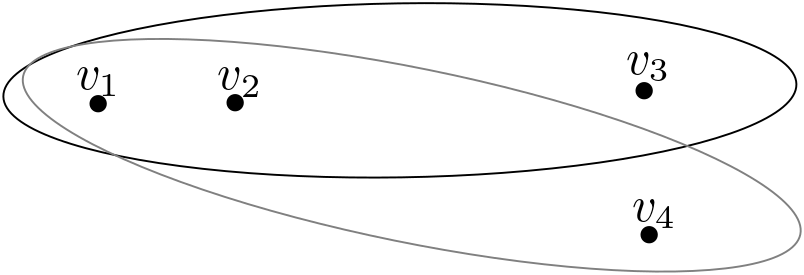}
    \caption{The hyperflower in Example \ref{ex:hyperflower}.}
    \label{fig:hyperflower}
\end{figure}\newline

%We have
    %\begin{align*}
     %   (K\mathbf{x}^2)_1&=2x_1^2 - x_2x_3 - x_2x_4\\
      %  (K\mathbf{x}^2)_2&=2x_2^2 - x_1x_3 - x_1x_4\\
       % (K\mathbf{x}^2)_3&=x_3^2-x_1x_2\\
        %(K\mathbf{x}^2)_4&=x_4^2-x_1x_2.
    %\end{align*}

The characterstic polynomial of $A$ is $(\lambda^3 - 4)^3\lambda^{23}$. Hence, the eigenvalues of $A$ are
\begin{align*}
    &0 \text{ with multiplicity 23} & \sqrt[3]{4} \text{ with multiplicity 3}\\
    &\sqrt[3]{4}\omega \text{ with multiplicity 3} 
    &\sqrt[3]{4}\omega^2 \text{ with multiplicity 3},
\end{align*}
where $\omega$ is a third root of $1$. In particular, the distinct eigenvalues are exactly the ones in Proposition \ref{prop:0flower} and Corollary \ref{cor:hyperflower}.\\
    The characteristic polynomial of $K$ is $(\lambda^2 - 5\lambda + 8)^3(\lambda - 1)^{13}(\lambda - 2)^{10}\lambda^3$. Its roots are
    \begin{align*}
        &0 \text{ with multiplicity 3}
        &\frac{5}{2}+\frac{\sqrt{7}}{2}\mathbf{i}\text{ with multiplicity 3}\\
        &\frac{5}{2}-\frac{\sqrt{7}}{2}\mathbf{i}\text{ with multiplicity 3}
        & 1 \text{ with multiplicity 13}\\
        & 2 \text{ with multiplicity 10}.
    \end{align*}
The characteristic polynomial of $\mathcal{L}$ is $(\lambda^2 - 3\lambda + 3)^3(\lambda - 1)^{23}\lambda^3$. Thus, its eigenvalues are
\begin{align*}
    & 0 \text{ with multiplicity 3}
    & 1 \text{ with multiplicity 23}\\
    & 2+\omega \text{ with multiplicity 3}
    & 2+\omega^2 \text{ with multiplicity 3}.
    %& \frac{3}{2}+\frac{\sqrt{3}}{2}i, \text{ with multiplicity 3}\\
    %& \frac{3}{2}-\frac{\sqrt{3}}{2}i, \text{ with multiplicity 3}.
\end{align*}
\end{example}
\begin{example}\label{ex:flower2}
    Let $G=(V,E)$ be the $3$-hyperflower with $V=\{v_1,v_2,v_3,v_4,v_5\}$ and $E = \{\{v_1,v_2,v_3\},\{v_1,v_2,v_4\},\{v_1,v_2,v_5\}\}$ (Figure \ref{fig:flower2}).
    \begin{figure}[ht]
    \centering
    \includegraphics[width=7cm]{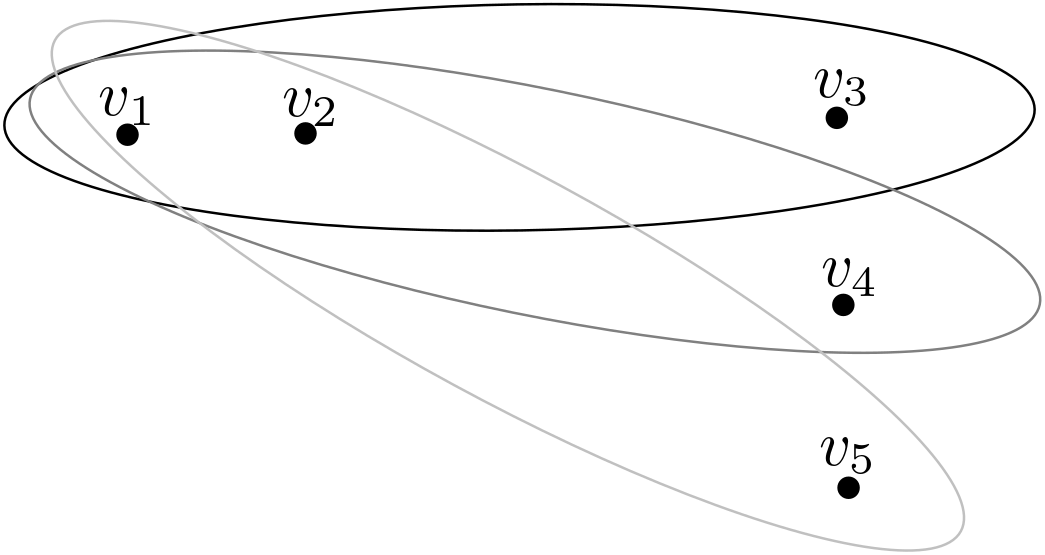}
    \caption{The hyperflower in Example \ref{ex:flower2}}
    \label{fig:flower2}
\end{figure}\newline
The characteristic polynomial of $A$ is $(\lambda^3 - 9)^3(\lambda^2 + \lambda + 1)^9(\lambda - 1)^9\lambda^{44}$,
therefore its eigenvalues are
\begin{align*}
    &0 \text{ with multiplicity 44}
    &1 \text{ with multiplicity 9}\\
    &\sqrt[3]{9} \text{ with multiplicity 3}
    &\sqrt[3]{9}\omega \text{ with multiplicity 3}\\
    &\sqrt[3]{9}\omega^2 \text{ with multiplicity 3}
    &\omega \text{ with multiplicity 9}\\
    &\omega^2 \text{ with multiplicity 9}.
    %&-\frac{1}{2} + \frac{\sqrt{3}}{2}i, \text{ with multiplicity 9 \textcolor{orange}{questo è $\omega$}}\\
    %&-\frac{1}{2} - \frac{\sqrt{3}}{2}i, \text{ with multiplicity 9 \textcolor{orange}{questo è $\omega^2$}}.
\end{align*}Also in this case, the distinct eigenvalues of $A$ are exactly the ones in Proposition \ref{prop:0flower} and Corollary \ref{cor:hyperflower}. Now, the characteristic polynomial of $K$ is
    \[
    (\lambda^3 - 7\lambda^2 + 15\lambda - 8)^9(\lambda^2 - 7\lambda + 15)^3(\lambda - 1)^{36}(\lambda - 3)^8\lambda^3.
    \]
The eigenvalues of $K$ are
\begin{align*}
    &0 \text{ with multiplicity 3}\\
    &0.7944305695994095 \text{ with multiplicity 9}\\
    &1 \text{ with multiplicity 36}\\
    &3 \text{ with multiplicity 8}\\
    &3.1027847152002956 + 0.6654569511528129\mathbf{i} \text{ with multiplicity 9}\\
    &3.1027847152002956 - 0.6654569511528129\mathbf{i} \text{ with multiplicity 9}\\
    &\frac{7}{2} + \frac{\sqrt{11}}{2}\mathbf{i}, \text{ with multiplicity 3}\\
    &\frac{7}{2} - \frac{\sqrt{11}}{2}\mathbf{i} \text{ with multiplicity 3}.
\end{align*}
The characteristic polynomial of $\mathcal{L}$ is 
\[ \frac{(9\lambda^3 - 27\lambda^2 + 27\lambda - 8)^9(\lambda^2 - 3\lambda + 3)^3(\lambda - 1)^{44}\lambda^3}{387420489}.
\]
Its eigenvalues are 
\begin{align*}
    & 0, \text{ with multiplicity 3}\\
    & 0.519250143230864, \text{ with multiplicity 9}\\
    & 1, \text{ with multiplicity 44}\\
    & 1.240374928384567 + 0.4163415888278001\mathbf{i}, \text{ with multiplicity 9}\\
    & 1.240374928384567 - 0.4163415888278001\mathbf{i}, \text{ with multiplicity 9}\\
    & 2+\omega, \text{ with multiplicity 3}\\
    & 2+\omega^2, \text{ with multiplicity 3}.
   % & \frac{3}{2}+\frac{\sqrt{3}}{2}i, \text{ with multiplicity 3}\\
    %& \frac{3}{2}-\frac{\sqrt{3}}{2}i, \text{ with multiplicity 3}.
\end{align*}
\end{example}

\section{Spectral symmetries}\label{section: symmetries}
%Purtroppo non ci sono nuove dimostrazioni in questa sezione perché ho scoperto che quelle che avevo fatto esistevano già (nel caso uniforme e nel caso generale sono identiche). Però penso che i risultati siano comunque interessanti e si possono usare per commentare gli esempi che abbiamo già.\\
In this section we discuss some spectral symmetries. First, we recall \cite[Theorem 3.12]{fan2019spectral} and we apply it to the hypergraph tensors.
\begin{definition}[\cite{fan2019spectral}]
Let $T$ be a tensor and let $\ell$ be a positive integer. The tensor $T$ is \emph{spectral $\ell$-symmetric} if 
\begin{equation*}
    \spec(T)=e^{\frac{2\pi\mathbf{i}}{\ell}}\spec(T).
\end{equation*}
\end{definition}
\begin{definition}[\cite{fan2019spectral}]Let $k\geq 2$ and $\ell\geq 2$ such that $\ell | k$. A $k$-th order $n$-dimensional tensor $T$ is \emph{$(k,\ell)$-colorable} if there exists a map $\phi:[n]\rightarrow [k]$ such that, if $T_{i_1,\ldots,i_k}\neq 0$, then
\begin{equation*}
    \phi(i_1)+\ldots+\phi(i_k)\equiv \frac{k}{\ell} \mod k.
\end{equation*}Such $\phi$ is an \emph{$(k,\ell)$-coloring} of $T$.
\end{definition}

\begin{definition}[\cite{fan2019spectral}]
The hypergraph $G$ is \emph{$(\nabla,\ell)$-colorable}, for some $\ell\geq 2$ such that $\ell | \nabla$, if there exists a map $\phi:V \rightarrow [\nabla]$ such that, if $\{v_{i_1},\ldots,v_{i_\nabla}\}\in E$, then
\begin{equation*}
    \phi(v_{i_1})+\ldots+\phi(v_{i_\nabla})\equiv \frac{\nabla}{\ell} \mod \nabla.
\end{equation*}
\end{definition}
\begin{remark}
Clearly, the hypergraph $G$ is $(\nabla,\ell)$-colorable if and only if its associated tensors are $(\nabla,\ell)$-colorable.
\end{remark}
\begin{theorem}[\cite{fan2019spectral}]\label{thm:ellsymm}
Let $T$ be a symmetric weakly irreducible non-negative tensor of order $k$. Then $T$ is spectral $\ell$-symmetric if and only if $T$ is $(k,\ell)$-colorable.
\end{theorem}

As an immediate consequence of Theorem \ref{thm:ellsymm}, we obtain the following
\begin{corollary}\label{cor:colorable}
A connected hypergraph $G$ is $(\nabla,\ell)$-colorable if and only if  one of $A$, $K^+$, $L^+$ and $\mathcal{L}^+$ are spectral $\ell$-symmetric, if and only if all of them are spectral $\ell$-symmetric.
\end{corollary}
\begin{proof}
The claim for $A$, $K^+$ and $L^+$ follows directly from Theorem \ref{thm:ellsymm}, since these are all symmetric non-negative tensors and, by Theorem \ref{thm:weaklyirr}, they are also weakly irreducible, as we are assuming that $G$ is connected. The claim for $\mathcal{L}^+$ then follows by Proposition \ref{prop:isospectral}.
\end{proof}

\begin{example}
    The hyperflowers from the previous section are $(\nabla,\nabla)$-colorable. This can be seen by coloring the central vertices with $\nabla$ and the peripheral vertices with $1$. By Corollary \ref{cor:colorable}, this implies that $A$ is spectral $\nabla$-symmetric, as we can observe in Examples \ref{ex:hyperflower} and \ref{ex:flower2}.
\end{example}

\begin{remark}
If $G$ is odd-bipartite and $V=V_1\sqcup V_2$ is a corresponding bipartition of the vertex set, then the map $\phi:V\rightarrow [\nabla]$ which has value $\frac{\nabla}{2}$ on $V_1$ and value $0$ on $V_2$ shows that $G$ is $(\nabla,2)$-colorable. 
\end{remark}

We now discuss another kind of spectral symmetry for the hypergraph tensors. Recall that a graph $G$ is \emph{bipartite} if one can decompose the vertex set as a disjoint union $V=V_1\sqcup V_2$ such that each edge has one endpoint in $V_1$ and one endpoint in $V_2$. It is known that, for a simple graph $G$, the following are equivalent:
\begin{enumerate}
    \item $G$ is bipartite
    \item $\lambda\in \spec(A) \iff -\lambda\in \spec(A)$, with the same multiplicity
\item $\spec(K)=\spec(K^+)$
\item $\lambda\in \spec(L) \iff 2-\lambda\in \spec(L)$, with the same multiplicity.
\end{enumerate}

\begin{remark}For simple graphs, being bipartite is the same as being $(2,2)$-colorable. Therefore, the equivalence of 1., 2.\ and 4.\ above (using Remark \ref{rmk2-}) is a special case of Corollary \ref{cor:colorable}.
\end{remark}

There are various ways of generalizing the notion of bipartite graph to the case of hypergraphs, see for instance the \emph{balanced hypergraphs} in \cite[Section 3]{balanced} or the \emph{bipartite hypergraphs} in \cite[Section 6.1.1]{JM2019}. We consider the \emph{odd-bipartite hypergraphs} in \cite{oddbipartite}, which are defined for uniform unweighted hypergraphs, and we generalize them for any hypergraph, as follows.
\begin{definition}
The hypergraph $G$ is \emph{odd-bipartite} if $\nabla$ is even and one can decompose the vertex set as a disjoint union $V=V_1\sqcup V_2$ such that, if $i_1,\ldots,i_\nabla \in [N]$, then
\begin{equation*}
    \{v_{i_1},\ldots,v_{i_\nabla}\}  \in E \Rightarrow  \text{ There is an odd number of vertices of }V_1 \text{ among } v_{i_1},\ldots,v_{i_\nabla},
\end{equation*}where the vertices $v_{i_1},\ldots,v_{i_\nabla}$ are counted with repetitions.
\end{definition}

Clearly, if $\nabla=2$, $G$ is a bipartite graph if and only if $G$ is odd-bipartite. Theorem 2.1 in \cite{oddbipartite}, which is formulated for $A$ and $K$ in the case of unweighted uniform hypergraphs, can be formulated also for $L$ and for all hypergraphs, as follows.
\begin{theorem}\label{thm:oddbipartite}
If $G$ is a connected hypergraph, the following conditions are equivalent:
\begin{enumerate}
    \item $\nabla$ is even and $G$ is odd-bipartite.
        \item There exists a diagonal matrix $P$ of order $N$ with all the diagonal entries $\pm 1$ and $P\neq -\id$ such that $A=-P^{-(\nabla-1)}AP$. 
    \item There exists a diagonal matrix $P$ of order $N$ with all the diagonal entries $\pm 1$ and $P\neq -\id$ such that $K=P^{-(\nabla-1)}K^+P$. 
    \item There exists a diagonal matrix $P$ of order $N$ with all the diagonal entries $\pm 1$ and $P\neq -\id$ such that $L=P^{-(\nabla-1)}L^+P$. 
\end{enumerate}
\end{theorem}Theorem \ref{thm:oddbipartite} can be proved as \cite[Theorem 2.1]{oddbipartite}. It can be applied for proving the following theorem, which generalizes Theorem 2.2 and Theorem 2.3 in \cite{oddbipartite}, as well as the known results for graphs that we discussed above.
\begin{theorem}\label{thm:oddbipartite2}
If $G$ is a connected hypergraph, the following conditions are equivalent:
\begin{enumerate}
    \item  $\nabla$ is even and $G$ is odd-bipartite.
     \item $\spec(A)=-\spec(A)$ and $\hspec(A)=-\hspec(A)$.
    \item $\hspec(A)=-\hspec(A)$.
    \item $\spec(K)=\spec(K^+)$ and $\hspec(K)=\hspec(K^+)$.
    \item $\hspec(K)=\hspec(K^+)$.
     \item $\spec(L)=\spec(L^+)$ and $\hspec(L)=\hspec(L^+)$.
    \item $\hspec(L)=\hspec(L^+)$.
\end{enumerate}
\end{theorem}
\begin{proof}
As Theorem 2.2 and Theorem 2.3 in \cite{oddbipartite}, using Theorem \ref{thm:oddbipartite} instead of \cite[Theorem 2.1]{oddbipartite}. 
\end{proof}
\begin{remark}
The second condition in Theorem \ref{thm:oddbipartite2} means that
\begin{equation*}
    \lambda \in \spec(A) \iff -\lambda \in\spec(A), \text{ with the same multiplicity,}
\end{equation*}and $\lambda \in \hspec(A) \iff -\lambda \in\hspec(A)$. Similarly, by Remark \ref{rmk2-}, the sixth condition in Theorem \ref{thm:oddbipartite2} means that 
\begin{equation*}
    \lambda \in \spec(L) \iff 2-\lambda \in\spec(L), \text{ with the same multiplicity,}
\end{equation*}and $
    \lambda \in \hspec(L) \iff 2-\lambda \in\hspec(L)$.
Moreover, by Proposition \ref{prop:isospectral}, the claim for $L$ and $L^+$ also holds for $\mathcal{L}$ and $\mathcal{L}^+$.
\end{remark}

%\begin{remark}By Corollary \ref{cor:colorable} and Theorem \ref{thm:oddbipartite2}, it is easy to see that odd-bipartite implies $(\nabla,2)$-colorable.\end{remark}

\section{Computing the eigenvalues}\label{sec: computations}
In this section we apply software implementing methods from numerical algebraic geometry to compute the set of eigenvalues of a tensor and their geometric multiplicity. In particular, this applies to all tensors associated to hypergraphs which were presented in the previous sections. %In order to be consistent with our treatment of the case of hypergraphs on $N$ vertices with $\nabla=k$, 
We shall consider here a $k$-th order $n$-dimensional tensor $T$. To such tensors we associate the set of solutions of the system of polynomial equations
\begin{align}\label{eq: system}
	\begin{cases}
	(T\mathbf{x}^{k-1})_1-\lambda x_1^{k-1}&=0\\
	&\vdots\\
	(T\mathbf{x}^{k-1})_n-\lambda x_n^{k-1}&=0.
	\end{cases}
\end{align}
 Here we use the symbols $x_1,\dots,x_n,\lambda$ as variables of our polynomials, and we will use $(\mathbf{x}^*,\lambda^*)\in\mathbb{C}^{n+1}$ for a specific solution. If $\lambda^*\in\spec(T)$, then we define the \emph{eigenvariety} of $\lambda^*$ as
 \[
 V(\lambda^*)=\{\mathbf{x}\in\C^n : (\mathbf{x},\lambda^*)\mbox{ is an eigenpair for }T\}.
 \]
The eigenvarieties are solutions of polynomial systems, so they are by definition \emph{algebraic varieties}. If we want to consider all eigenvarieties together, then it is convenient to look at their union in the bigger space $\C^{n+1}$, where $\lambda$ is considered a variable as well. Hence we define
\[
V=\{(\mathbf{x},\lambda)\in\C^{n+1} : (\mathbf{x},\lambda)\mbox{ is an eigenpair for }T\}
\]
to be the set of solutions of \eqref{eq: system}. In other words,
% The set of solutions of \eqref{eq: system} is by definition an \emph{algebraic variety}, which we denote by V\subseteq \mathbb{C}^{n+1}$%, and the finite multiset of values of $\lambda$ for which the equation admits a solution $(\mathbf{x}^*,\lambda^*)$, with $\mathbf{x}^*\neq 0$ is the spectrum of $T$
 %. {\color{red}We can} then decompose $V$ into a finite union
\begin{equation*}%\label{eq: decompose V}
	V=\{(0,\dots,0,\lambda): \lambda\in\mathbb{C}\}\cup\bigcup_{\lambda^*\text{ eigenvalue of } T} V(\lambda^*)\times\{\lambda^*\}.
\end{equation*}
%where $V(\lambda^*)\subseteq\mathbb{C}^n$ is called the  \emph{eigenvariety}  of $\lambda^*$. 
The choice of considering $\lambda$ as a variable comes from our application-oriented approach. Indeed, $\lambda$ is a variable for the software.\newline
In general, eigenvarieties are not \emph{irreducible}, i.e., they can be further decomposed into a union of varieties, of possibly different dimensions. We shall see an example of this behaviour. For a fixed $\lambda^*$, 
the equations $(T\mathbf{x}^{k-1})_i-\lambda^* x_i^{k-1}=0$ are homogeneous of degree $k-1$. In particular, if $\mathbf{x}^*\in\mathbb{C}^n$ is an eigenvector of $T$ with eigenvalue $\lambda^*$, then $c \mathbf{x}^*$ is also an eigenvector with eigenvalue $\lambda^*$, for every $c\in\mathbb{C}^*$. In geometric terms, $V(\lambda^*)\subseteq \mathbb{C}^n$ is a \emph{cone}. When we deal with matrices, all eigenvarieties are linear spaces. However, for $k\geq 3$ eigenvarieties can have higher degrees. \\
Our main goal is to compute the eigenvalues of $T$. When the entries of $T$ can be represented on a field with exact arithmetic (for instance when they are rational numbers) we can   use \emph{Gr\"{o}bner basis} techniques to compute the characteristic polynomial of $T$, defined in Section \ref{sec: intro}.%which is the resultant of the system of equations $(T\mathbf{x}^{k-1})_i=\lambda x_i^{k-1}$
 For instance, the package ``Resultants" of the computer algebra system Macaulay2 implements this idea. However, due to the notorious complexity of computing a Gr\"{o}bner basis, computations do not terminate even for rather small hypergraphs.\\
This motivates us to consider numerical methods to solve the polynomial system. In the last decade a lot of research has been carried out in developing methods based on the so-called \emph{homotopy continuation}. The idea is that in order to solve a system of polynomial equations $S$, one can first transform it into a similar, but simpler to solve, system $S'$. After solving $S'$, solutions of $S$ can be tracked back via a \emph{homotopy} between the two polynomial systems. There are several packages and sofware implementing this idea, such as Bertini and PHCpack. In the example below we use the package HomotopyContinuation.jl \cite{HC.jl}, developed in the language Julia, and we refer to the webpage \href{https://www.juliahomotopycontinuation.org/}{juliahomotopycontinuation.org} for documentation, examples and applications.\newline
In this way we are able to compute the geometric multiplicity of each eigenvalue, with high probability. However, it is not immediate to apply this technique to compute the span multiplicity nor the algebraic multiplicity.\newline
The software does not accept positive dimensional systems as input, and it suggests to add a generic affine linear equation to the system. We choose a generic affine linear polynomial $\ell\in \mathbb{C}[x_1,\dots,x_n]$. Observe that $\ell$ does not involve the variable $\lambda$. In this way the affine hyperplane defined by $\ell$ does not intersect the line $\{(0,\dots,0,\lambda): \lambda\in\mathbb{C}\}$. However, given the particular structure of our system, we can guarantee that there is a Zariski open subset of the space of affine linear polynomials $\ell$ of $\mathbb{C}[x_1,\dots,x_n]$ such that the hyperplane $\{(\mathbf{x},\lambda): \ell(\mathbf{x},\lambda)=0\}\subseteq \mathbb{C}^{n+1}$ intersects all irreducible components of all varieties $V(\lambda^*)\times \{\lambda^*\}$. 
%\begin{lemma}
%	There is a Zariski dense open subset of the space of affine linear polynomials $\ell$ of $\mathbb{C}[x_1,\dots,x_n]$ such that the hyperplane $H_{\ell}=\{(\mathbf{x},\lambda): \ell(\mathbf{x},\lambda)=0\}$ has non empty intersection with all irreducible components of $V(\lambda^*)\times \{\lambda^*\}$, for every $\lambda^*\in \text{Spec}(T)$.\end{lemma}\begin{proof}	Let $\ell$ be a generic affine linear polynomial in $\mathbb{C}[x_1,\dots,x_n]$. Let now $W$ be an irreducible component of $V(\lambda^*)$. As $W$ is a conic variety, it contains at least a line $L$. The hyperplane $\{\ell(\mathbf{x})=0\}$ does not intersect $L$ if and only if $L$ is contained in the translate of $\{\ell(\mathbf{x})=0\}$ containing the origin. This is a Zariski closed condition.\end{proof}
\begin{lemma}
The general affine hyperplane of $\C^n$ intersects every irreducible component of $V(\lambda^*)$ for every $\lambda^*\in$ Spec$(T)$.
\end{lemma}
\begin{proof}
Let $W$ be an irreducible component of $V(\lambda^*)$. As $W$ is a cone, it contains at least a line $L$. An affine hyperplane $H$ does not intersect $L$ if and only if $L$ is contained in the translate of $H$ containing the origin. This is a Zariski closed condition.
\end{proof}

\begin{example}\label{ex: figure A}
	Let $G=(V,E)$ be the unweighted hypergraph with $V=\{v_1,\ldots,v_5\}$ and $E = \{\{v_1,v_2,v_3,v_4\},\{v_1,v_5\},\{v_3,v_5\}\}$. This is the protein--protein interaction network depicted in \cite[Figure 1 A]{klamt2009hypergraphs}. \newline 
	\begin{figure}[ht]
		\centering
		\includegraphics[width=7cm]{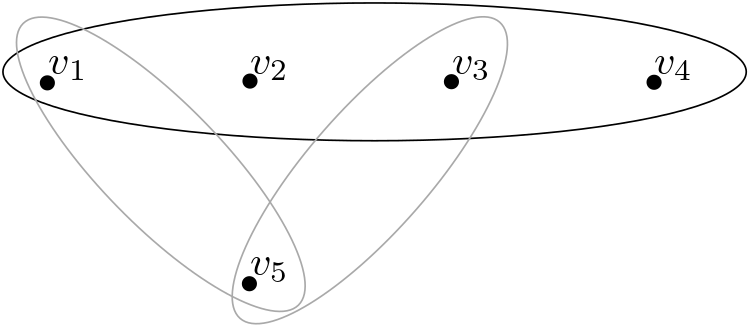}
		\caption{The hypergraph in Example \ref{ex: figure A}}
		\label{fig:Theis}
	\end{figure}\newline 
	As $\nabla=4$, $A,K,L,\mathcal{L}$ are all tensors in $\mathbb{C}^5\otimes \mathbb{C}^5 \otimes \mathbb{C}^5$. We consider the adjacency tensor $A$, i.e., we compute the solutions of \eqref{eq: system} with $T=A$. On a standard desktop computer the command {\texttt{resultant}} in Macaulay2 does not terminate the computation. This is perhaps not surprising, if we think that the characteristic polynomial has degree $5\cdot 3^4=405$.\\
	We then use the software \texttt{HomotopyContinuation.jl}, which can be called via the command \texttt{using HomotopyContinuation} in Julia. We start generating a random affine linear polynomial $p$ in $\mathbb{C}[x_1,\dots,x_n]$ and append it to the system \eqref{eq: system}.
	\begin{align*}
		\mathtt{p = sum(rand(Complex{Float64},5).*[x1,x2,x3,x4,x5])+rand(Complex{Float64})}
	\end{align*}
	 The syntax to define our system in \texttt{HomotopyContinuation.jl} is the following:
	\begin{align}\label{eq: system julia}
		&\mathtt{@var\enspace x1\enspace x2\enspace x3\enspace x4\enspace x5\enspace l;}\nonumber\\
		&\mathtt{f = System([-l*x1^3 + x2*x3*x4 + 3/7*x1^2*x5 + 3/7*x1*x5^2 + 1/7*x5^3,}\nonumber\\
		& \mathtt{-l*x2^3 + x1*x3*x4,}\\
		& \mathtt{-l*x3^3 + x1*x2*x4 + 3/7*x3^2*x5 + 3/7*x3*x5^2 + 1/7*x5^3,} \nonumber\\
		& \mathtt{-l*x4^3 + x1*x2*x3,}\nonumber\\
		& \mathtt{ -l*x5^3 + 1/7*x1^3 + 1/7*x3^3 + 3/7*x1^2*x5 + 3/7*x3^2*x5 + 3/7*x1*x5^2 + 3/7*x3*x5^2,}\nonumber\\
		& \mathtt{p]);}\nonumber
	\end{align}
	Here we use the symbol $\mathtt{l}$ in place of $\lambda$. We run the command \texttt{result = solve(f);} which stores the solution in a variable called ``result". The computation terminates in few seconds and the program prints the following information.
	
	\begin{figure}[h]
	    \centering
		\includegraphics[scale=0.8]{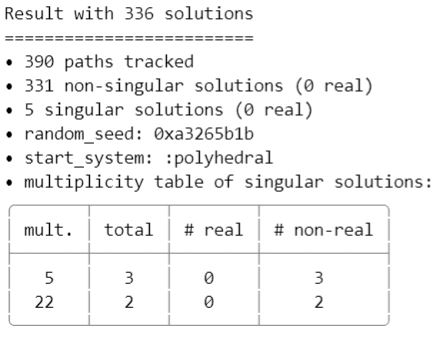}
	\end{figure}
	
	Let us comment on the output. The program found $331$ non-singular solutions $(\mathbf{x}^*,\lambda^*)\in\mathbb{C}^{n+1}$ and $5$ singular solutions. A solution is marked as singular using two classical parameters in complex analysis: the \emph{winding number} and the \emph{condition number} of the Jacobian of the system. In particular, singular solutions approximate solutions on a positive dimensional component of $V$, or solutions on a line which has multiplicity greater than $1$ in $V$. None of the found solutions are marked as real. It does not mean that there is no solution $(\mathbf{x}^*,\lambda^*)\in\mathbb{R}^{n+1}$, but simply that the eigenvectors corresponding to real eigenvalues found by the software have some complex coordinate. This is due to the choice of the polynomial $p\in\mathbb{C}[x_1,\dots,x_5]$. The reader concerned with real eigenpairs can apply the same procedure using generic polynomials in $\mathbb{Q}[x_1,\dots,x_5]$.\\
	We remark that certain steps in the solution process involve randomness, hence calling the function ``solve'' twice on the same starting system might lead to a different number of solutions. However, as we are interested in computing the eigenvalues, it suffices to find at least one point in each eigenvariety.\\
	To retrieve the non-singular solutions we simply type \texttt{solutions(result)}, which returns an array of $331$ vectors in $\mathbb{C}^6$. We are interested in the value of the coordinate $\lambda$ of each vector. These numbers are not all distinct in general. We can easily write a loop to discard those eigenvalues which are close to another one. We obtain a list of $64$ complex numbers which we plot in the complex plane (Figure \ref{fig: plot}). By further selecting only the real part of those with a sufficiently small imaginary part ($<10^{-15}$) we obtain the following $18$ real eigenvalues.
	\begin{center}
		\begin{tabular}{ l l }
		-1.1503540417366391 & $1.7219155529623352\cdot 10^{-41}$\\
		-1.0589738102553747 & 0.9382912060665167\\
		-1.0 & 0.9858713918602654\\
		-0.9233845418913038 & 1.0\\
		-0.5474615312663447 & 1.071873232613355\\
		-0.1532944068758618 & 1.0858832885825462\\
		-0.1484156441177043 & 1.2267760851792766\\
		-0.14163743538075982& 1.4284010786135974\\
		-0.25425744432182457& 1.73405913985699\\
		\end{tabular}
	\end{center}
	\begin{figure}[h]
		\includegraphics[scale=0.65]{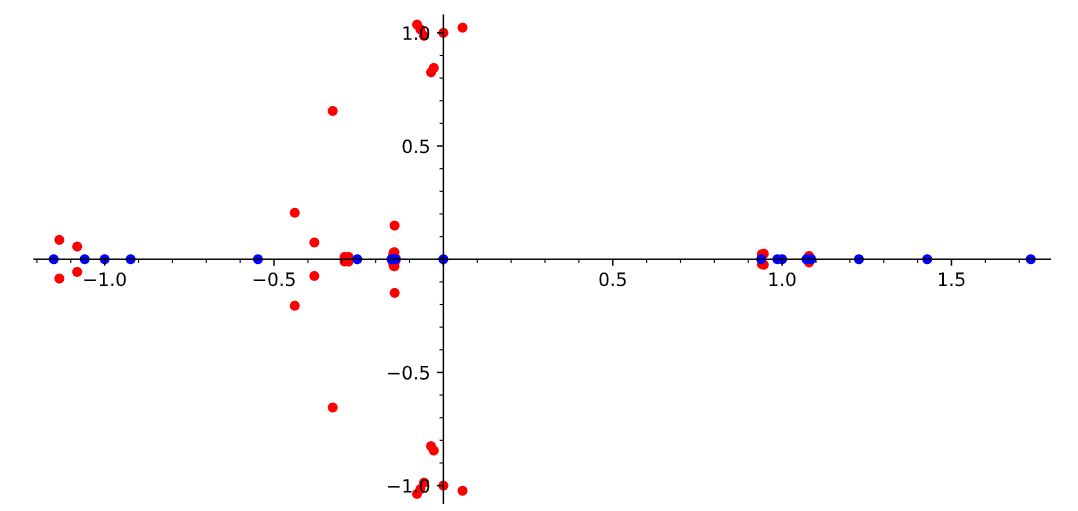}
		\caption{The $64$ eigenvalues for the adjacency tensor of the hypergraph in Example \ref{ex: figure A}, found by numerically solving the system of equations with \texttt{HomotopyContinuation.jl}. The blue points have an imaginary part smaller than $10^{-15}$.  }
		\label{fig: plot}
	\end{figure}
	The singular solutions $(\mathbf{x}^*,\lambda^*)$ in this example, which can be accessed typing 
	\begin{center}
	   \texttt{singular(result)}, 
	\end{center}
	 all satisfy $\lambda^*=0$.
	\end{example}
	The numerical solution can be employed to obtain information on the geometric multiplicity of eigenvalues. The idea is the following: we generate new generic linear polynomials $p_i\in\mathbb{C}[x_1,\dots,x_n]$, append the polynomials $p_1,\dots,p_k$ to \eqref{eq: system} and solve the system numerically. Observe that unlike $p$, the polynomials $p_i$ are not affine linear, i.e., there is no degree $0$. 
	
	\begin{lemma}
		Let $p\in\mathbb{C}[x_1,\dots,x_n]$ be an affine linear polynomial, and for $1\leq i\leq k$ let $p_i\in\mathbb{C}[x_1,\dots,x_n]$ be linear polynomials. If $(\mathbf{x}^*,\lambda^*)\in\mathbb{C}^{n+1}$ is a solution of the system
		\begin{align*}
		\begin{cases}
		&(T\mathbf{x}^{k-1})_1-\lambda x_1^{k-1}=0\\
		&\vdots\\
		&(T\mathbf{x}^{k-1})_n-\lambda x_n^{k-1}=0\\
		&p(\mathbf{x},\lambda)=p_1(\mathbf{x},\lambda)=\cdots=p_k(\mathbf{x},\lambda)=0,
		\end{cases}
		\end{align*}
		then $\lambda^*$ is an eigenvalue with $\text{gm}(\lambda)\geq k+1$. If after adding a new generic linear polynomial $p_{k+1}$ to the system there is no solution with $\lambda=\lambda^*$, then $\text{gm}(\lambda)= k+1$.
	\end{lemma}
	Therefore, we can continue adding generic linear polynomials until no solutions are found. All eigenvalues are found in this way, with the correct geometric multiplicity. We illustrate this fact in Example \ref{example: julia}.
	\begin{example}\label{example: julia}
		As we pointed out in Example \ref{ex: figure A}, the unique eigenvalue obtained from a singular solution of \eqref{eq: system julia} is $\lambda^*=0$. Therefore this is the only eigenvalue which might have geometric multiplicity higher than $1$. To verify that this the case, we append to the system \eqref{eq: system julia} a random linear polynomial $p_1\in\mathbb{C}[x_1,\dots,x_5]$ and compute a solution numerically. We find solutions with $\lambda^*=0$, hence $\text{gm}(0)\geq 2$. Moreover, the system obtained adding another random linear polynomial $p_2$ does not admit any solution. We conclude that $\text{gm}(0)=2$, and all the other eigenvalues have geometric multiplicity $1$.
	\end{example}

\subsection*{Funding}
FG is supported by the National Science Center, Poland, project ``Complex contact manifolds and geometry of secants'', 2017/26/E/ST1/00231, and acknowledges partial support by the fund FRA 2018 of University of Trieste -- project DMG, funded by the MIUR Excellence Department Project awarded to the Department of Mathematics and Geosciences, University of Trieste. RM is supported by The Alan Turing Institute under the EPSRC grant EP/N510129/1. LV is funded by the G\"oran Gustafsson foundation.

\subsection*{Acknowledgments}We are grateful to the anonymous referee for the comments and suggestions that have greatly improved the first version of this paper.  We thank Christian Kuehn (TUM) for the helpful comments and suggestions. We would like to thank Bernd Sturmfels (MPI MiS), who has always actively encouraged us to work together. 

\bibliography{Tensors}

\end{document}